\documentclass[11pt]{article}
\usepackage{amsfonts,latexsym}
\usepackage{amssymb}
\usepackage{amsmath,amsthm}
\usepackage{enumerate}
\usepackage{framed}
\usepackage{verbatim}
\usepackage[]{algorithm2e}
\usepackage{algorithm}
\usepackage{algorithmic}
\usepackage{color}
\usepackage{tikz}

\usepackage[toc,page]{appendix}

\newtheorem{thm}{Theorem}[section]

\theoremstyle{definition}
\newtheorem{defn}[thm]{Definition}
\newtheorem{lem}[thm]{Lemma}

\newtheorem{tm}[thm]{Theorem}
\newtheorem{rem}[thm]{Remark}

\theoremstyle{remark}

\newtheorem*{notn}{Notation}

\renewcommand{\P}{\mathbb{P}}

\title{Packing Directed Hamilton Cycles Online}
\author{ Michael Anastos\footnote{email:manastos@andrew.cmu.edu} and Joseph Briggs\footnote{email:jbriggs@andrew.cmu.edu} \vspace{5mm}
\\Department of Mathematical Sciences,
\\Carnegie Mellon University,
\\Pittsburgh PA 15213.
\date{}}

\begin{document}

\maketitle

\begin{abstract}
    
    Consider a directed analogue of the random graph process on $n$ vertices, where the $n(n-1)$ edges are ordered uniformly at random and revealed one at a time.
    It is known that w.h.p.\@ the first digraph in this process with both in-degree and out-degree $\geq q$ has a $q$-edge-coloring with a Hamilton cycle in each color. We show that this coloring can be constructed online, where each edge must be irrevocably colored as soon as it appears.
In a similar fashion, for the \emph{undirected} random graph process, we present an online $n$-edge-coloring algorithm which yields w.h.p.\@ $q$ disjoint rainbow Hamilton cycles in the first graph containing $q$ disjoint Hamilton cycles.
\end{abstract}
\section{Introduction}
  Let $\vec{K}_n$ be the complete directed graph on $n$ vertices. We let $(e_1,e_2,...,e_{n(n-1)})$  be a uniformly random permutation of the edges of $\vec{K}_n$ and consider the random process of digraphs $D_1,D_2,...,D_{n(n-1)}$ defined by $D_m=(V_n,E_m)$ with  $E_m=(e_1,...,e_m)$ for $m \in [n(n-1)].$ This is a directed analogue of the celebrated Erd\H{o}s-R\'{e}nyi random graph process \cite{ER}, in which the edges of the \emph{undirected} complete graph $K_n$ are ordered uniformly at random, similarly yielding a random process of graphs $ G_1, G_2, \dots, G_{n(n-1)/2}=K_n$.
  Graph-theoretic properties of $D_m$ and $G_m$ are said to hold ``with high probability'' (w.h.p.\@) if they occur with probability $1-o(1)$ as $n \rightarrow \infty$, where $m$ is allowed to be a random variable depending on $n$.

A \emph{Hamilton cycle} is a (directed) cycle passing through all $n$ vertices exactly once. When a graph or digraph contains such a cycle, we say it is \emph{Hamiltonian}.
 The study of Hamilton cycles is fundamental to graph theory, including in the random setting.
For a digraph to contain a Hamilton cycle it certainly requires each vertex to have at least 1 in-edge and 1 out-edge, but quite remarkably, this is almost always sufficient for the random graphs $D_m$.
Specifically,    for a fixed $q$, let $D_{\tau_q}$ denote the first digraph in this random process with 
both minimum in-degree and out-degree $\geq q$.
In \cite{apaper}, Frieze showed that w.h.p.\@  $D_{\tau_1}$ is Hamiltonian
yielding a hitting-time strengthening of McDiarmid \cite{dirp} and a directed version of the classical result due to Bollob\'{a}s \cite{Evolutionhit2} and Ajtai, Koml\'{o}s and Szemer\'{e}di \cite{hit1}. The latter two papers independently proved that w.h.p.\@ the first $G_{m}$ in the undirected random graph process with minimum degree $\delta(G_m) \geq 2$ is Hamiltonian, thus bringing to fruition the work built up by Koml\'{o}s and Szemer\'{e}di \cite{Komlos}, Korshunov \cite{earlier1} and P\'{o}sa \cite{earlier2} previously.

The undirected version was strengthened  \cite{hitconstant} by Frieze and Bollob\'{a}s to additional Hamilton cycles thus: let $q=O(1)$ be fixed.
Let $G_{\tau_{2q}'}$ be the first random graph in the undirected process with $\delta(G_{\tau_{2q}'}) =2q$ ( here $\tau_q$ and $\tau_q'$ distinguish the directed and undirected hitting times respectively).
Then w.h.p.\@ $G_{\tau_{2q}'}$ has a $q$-edge-coloring with a Hamilton cycle in every color.
In fact, results for $q \rightarrow \infty$ with $n \rightarrow \infty$ have been established in all cases thanks to extensive work completed by Knox, K\"{u}hn and Osthus \cite{manyhit} and Krivelevich and Samotij \cite{manyhit2}.

In these papers, it appeared that the minimum degree conditions were still the most binding aspects of the proofs, suggesting stronger results could be obtained if corresponding minimum degree conditions are met.
Indeed, Krivelevich, Lubetzky and Sudakov \cite{Achlioptas} took advantage of the Achlioptas process with parameter $K=o(\log n)$ to build a Hamilton cycle using w.h.p.\@ only $(1+o(1))\frac{\tau_2'}{K}$ edges. In this process, at each time step, $K$ random new edges are presented, out of which one is added to the current graph, thereby allowing a bias towards low-degree vertices when necessary.
 In a similar fashion,
  Briggs, Frieze, Krivelevich, Loh and Sudakov \cite{pp} extended the classical result to an on-line version. 
  They presented an algorithm coloring the edges $(e_1,e_2,...,e_{{n(n-1)}/{2}})$ as they appeared, with $q=O(1)$ colors, such that w.h.p.\@ $G_{\tau_{2q}'}$ contains a monochromatic Hamilton cycle of every color.
  The on-line nature of this coloring is of importance, because the color of each new random edge $e_m$ \emph{cannot depend on the location of the edges appearing thereafter}.

In this paper we consider the analogous scenario in the directed random graph process. Here, the edges of the random permutation $(e_1,e_2,...,$ $e_{n(n-1)})$ of $\vec{K}_n$ are revealed one by one. As soon as an edge is revealed it has to be colored irrevocably with one of $q=O(1)$ colors.
We prove the following:

\begin{tm}\label{maintm}
There exists an on-line $[q]$-edge-coloring algorithm for $D_1, \dots, D_{n(n-1)}$ such that {w.h.p.\@}
 $D_{\tau_q}$ has  $q$ monochromatic Hamilton cycles, one in every color in $[q]$.
\end{tm}
In order to prove Theorem \ref{maintm} we present a coloring algorithm which we name $COL$. Thereafter we split the proof into two parts. In the first part we prove that each color class $c$ of $D_{\tau_q}$ given by $COL$ satisfies the minimum degree condition necessary for Hamiltonicity.
In the second part (drawing our proof strategy from \cite{apaper})
we fix $c\in [q]$ and show w.h.p.\@  $D_{\tau_q}$ has a monochromatic Hamilton cycle in color $c$. To do so we end up giving a reduction to the following more general Lemma.
\begin{lem}\label{reduction}
Let $F,H,D_{n,p}$ be digraphs on the same vertex set of size $n$ such that:
\begin{enumerate}[i)]
\vspace{-2mm}
\item $F$ is a 1-factor consisting of $O(\log n)$ directed cycles, 
\vspace{-2mm}
\item $H$ has maximum in/out-degree $O(\log n)$,
\vspace{-2mm}
\item $D_{n,p}$ is a random graph where every edge appears independently with probability $p=\Omega(\frac{\log n}{n})$.
\end{enumerate}
Then w.h.p.\@ there is  a Hamilton cycle spanned by $E(F)\cup \big(E(D_{n,p})\setminus E(H)\big)$.  
\end{lem}  
In \cite{twice}, Lee, Sudakov and Vilenchik also considered the on-line undirected random graph process. They were \emph{orienting} each new edge $\{ u,v \}$ as either the \emph{directed} edge $u \rightarrow v$ or $v\rightarrow u$, to form a \emph{directed} cycle in $G_{\tau_2'}$ (as opposed to coloring edges as they appear).
It turns out that the techniques that we use in order to prove Theorem \ref{maintm} can be used to prove a combination of \cite{pp} and \cite{twice}, namely:
\begin{tm}\label{secondtm}
There exists an on-line algorithm that orients \emph{and} $[q]$-edge-colors $G_1, \dots , G_{n(n-1)/2}$ such that w.h.p.\@ $G_{\tau_{2q}'}$ has $q$ {\bf{\emph{{directed}}}} Hamilton cycles, one in every color in $[q]$.
\end{tm}
Since the proofs of Theorems \ref{maintm} and \ref{secondtm} are almost identical we will not give a detailed proof of Theorem \ref{secondtm}. Instead we provide the algorithm and the main difference in the appendix. 
\vspace{3mm}
\\A beautiful consequence of Theorem \ref{secondtm} is the following. A Hamilton cycle is \emph{rainbow} if it does not contain two edges of the same color. Ferber and Krivelevich proved in \cite{rainbow2} that for  $p=(\log n + \log \log n+\omega(1))/n$ if we color uniformly at random the edges of $G_{n,p}$ with $(1+o(1))n$ colors, then the resulting graph w.h.p.\@ contains a rainbow Hamilton cycle,
improving previous results of Frieze and Loh \cite{FL} following Cooper and Frieze \cite{CF}.
Unfortunately, we cannot replace $(1+o(1))n$ colors with $n$. Indeed, among $n$ colors assigned to $\sim \frac{1}{2}n \log n$ edges, there is w.h.p.\@ some color that never appears, so there is no hope of a rainbow Hamilton cycle.
By contrast, in our (slightly) more deterministic on-line setting, we have that $n$ colours are indeed sufficient:
\begin{tm}\label{thirdtm}
There exists an on-line algorithm that orients \emph{and} $[n]$-edge-colors $G_1, \dots, G_{n(n-1)/2}$ such that $G_{\tau_{2q}'}$ has w.h.p. $q$ edge-disjoint {\bf{\emph{{directed rainbow}}}} Hamilton cycles. In particular, $G_{\tau_2'}$ has a rainbow Hamilton cycle (upon ignoring the directions).
\end{tm}
Indeed, given an algorithm  $COL$-$ORIENT$ satisfying Theorem \ref{secondtm}, we can construct an algorithm $COL$-$RBOW$ that satisfies Theorem \ref{thirdtm} in the following way. Write $V=\{v_1,...v_n\}$,
and whenever $COL$-$ORIENT$ directs an edge from $v_i$ to $v_j$, let $COL$-$RBOW$ color it $i$. At time $\tau_{2q}'$, any directed Hamilton cycle given by $COL$-$ORIENT$ has distinct out-vertices for every edge, and therefore distinct colors given by $COL$-$RBOW$.
So, $COL$-$ORIENT$ yielded $q$ edge-disjoint Hamilton cycles w.h.p.\@, and $COL$-$RBOW$ gave them all rainbow colors.
\vspace{3mm}
\\Throughout the paper we use the well-known result (see for example \cite{abook})  that w.h.p.
$$ n\log n+ n(q-1)\log \log n -\omega \leq \tau_q , \tau_q' \leq  n\log n+ n(q-1)\log \log n +\omega $$
for any  $\omega = \omega(n)$ which tends to infinity as $n$ tends to infinity.
\section{The Colouring Algorithm $COL$}
The coloring algorithm $COL$,  given shortly, will color greedily arcs that are incident to vertices that ``do not see all the colors yet''. These vertices are the most dangerous, as indeed some will only have $q$ out-arcs in $D_{\tau_q}$, accordingly needing \emph{exactly} 1 of each color.
We formalize these most needy of vertices by means of the notation in the following subsection, to guide our description of the algorithm $COL$. Note that the notation given below will be used repeatedly throughout the paper.

\subsection{Some notation}\label{notation}
\begin{notn}
``By/at time t'' is taken to mean ``after t edges have been revealed'', that is, with respect to $D_t$. We also write $\tau$ for $\tau_q$.
\end{notn}
\begin{defn}
For $v \in V_n$, $c \in [q]$ and $t \in \{0,1,...,\tau\}$, we set $d^+_t(v,c)$ (and $d^-_t(v,c)$ resp.\@) to equal the numbers of arcs with out-(in- resp.\@) vertex $v$, that have been revealed by time $t$ and have been assigned color $c$ by the algorithm COL.  Also write $d^+_t(v)$ ($d^-_t(v)$ resp) for the total number of out- (in-) arcs from $v$ by time $t$. Hence $d^+_t(v)=\sum_{c \in [q]}d^+_t(v,c) $  at any time $t$. For the final in/out-degrees we write $d^-(v):=d_\tau^-(v)$ and $d^+(v):=d_\tau^+(v)$.
\end{defn}
\begin{defn}
For $v \in V_n$ and $t \in \{0,1,...,\tau\}$ we set $C^+_v(t)$:=$\{ c \in [q] : d^+_t(v,c)=0 \}$ (i.e the colors that at time $t$ are missing from the out-arcs of $v$). Similarly set $C^-_u(t)$:=$\{ c \in [q] : d^-_t(v,c)=0 \}$.
\end{defn}
\begin{defn}
For $t \in \{0,1,...,\tau\}$ we set $FULL^+_t$:=$\{v \in V_n: C^+_v(t) =\emptyset \}$ (i.e.\@ the set of vertices that at time $t$ have out degree in each color at least one).  Similarly define $FULL^-_t$.
We certainly want both $FULL^+_\tau, FULL^-_\tau$ to contain all of $V_n$ in the end.
\end{defn}
\subsection{Algorithm  $COL$}
Algorithm ColorGreedy($u,v,t$) will be called in multiple places during the algorithm $COL$, hence is given beforehand.
\begin{algorithm}
\caption{ColorGreedy($u,v,t$)}
\uIf{ $u \notin FULL^+_{t-1}$ or $v \notin FULL^-_{t-1}$}{
 color arc $uv$ by a  color  chosen uniformly at random from   $C_u^+(t-1) \cup C_v^-(t-1)$.}
\Else{ color arc $uv$ by a  color  chosen uniformly at random from  $[q]$.}
\end{algorithm}
\\
For $i \in \{0,1,2,3\}$ we also set $m_i=i \cdot e^{-q\cdot 10^{4}} n \log n$,  marking out 3 small but positive fractions of the (expected) number of edges $\tau$, and $p_i=\frac{m_i}{n(n-1)}$.
\begin{algorithm}[H]
\caption{COL}
\For{ $t=1,...,m_1$}{
let $e_t=uv$ \\
Execute ColorGreedy($u,v,t$).}
For $v \in V_n$ set $c^+(v)=1$, $c^-(v)=1$.\\
\For{$t=m_1+1,m_1+2,...,m_2$}{
let $e_t=uv$ \\
\eIf{ $ u\notin FULL^+_{t-1} \text{ or } v \notin FULL^-_{t-1}$}{
Execute ColorGreedy($u,v,t$).}{  Color  the arc $uv$ by the color $c$  satisfying $c\equiv c^+(u) \mod q$,\\
$c^+(u) \leftarrow c^+(u)+1$.
}}
\For{$t=m_2+1,m_2+2,...,m_3$}{
let $e_t=uv$ \\
\eIf{ $ u\notin FULL^+_{t-1} \text{ or } v \notin FULL^-_{t-1}$}{
Execute ColorGreedy($u,v,t$).}{ Color  the arc $uv$ by the color $c$  satisfying $c\equiv c^-(v) \mod q$,\\
$c^-(v) \leftarrow c^-(v)+1$.
}}
For $i \in \{1,2,3\}, * \in \{+,-\}$ set $B^*_i$:=$\{v \in V_n: d^*_{m_i}(v)-d^*_{m_{i-1}} \leq \epsilon \log n\},$ where $\epsilon=e^{-q \cdot 10^{6}}$.
\\Furthermore set $BAD$:=$B_1^+ \cup B_1^- \cup B_2^+ \cup B_3^-$ and
$E':=\emptyset$.\\
\For{ $t=m_3+1,...,\tau$}{
let $e_t=uv$ \\
	\uIf{$ u\notin FULL^+_{t-1} \text{ or } v \notin FULL^-_{t-1}$}{
    Execute ColorGreedy($u,v,t$) \;
  }
  \uElseIf{$u \in BAD$ or $v \in BAD$}{
    Color  the arc $uv$ by a  color $c$ that minimizes  
    $d^+_t(u,c)\mathbb{I}(u \in BAD)+  d^-_t(v,c)\mathbb{I}(v \in BAD) .$  If there is more than one such color then choose one from them uniformly at random. \;
  }
  \Else{
    Execute ColorGreedy($u,v,t$).
  \\  Add the arc $uv$ to $E'$.
  }
  }
\end{algorithm}

\begin{rem}\label{rem1}
 Suppose at some time $t$ that $e_t=uv$ and ${C_u^+}(t-1) \cup {C_v^-}(t-1) \neq \emptyset $, i.e.\ $u \not\in FULL^+_{t-1}$ is still missing an out-edge color or $v \not\in FULL_{t-1}^-$ is still missing an in-edge color.  Then any color from ${C_u^+}(t-1) \cup {C_v^-}(t-1)$ has probability at least $\frac{1}{q}$ to be chosen to color $uv$.
\end{rem}
\begin{rem}\label{rem2}
The second priority (after the vertices needing to be greedy) is to build the 1-factor $F$ in each color needed to power Lemma \ref{reduction}, for which we aim to have as many vertices with at least a prescribed out-degree as possible (in fact, 6 will do).
The cycling with $c^+$ and $c^-$ between edge colors during times $(m_1,m_2]$ and $(m_2, m_3]$ will ensure as many of the $FULL$ vertices as possible receive an ample balance of edges in each color. The few exceptions are confined to $BAD$ and forced to balance their colors for the remainder of the process.
 Meanwhile, the arcs in $E'$ enjoy full randomness, and can be used to build the desired Hamilton cycles using classical techniques. 
\end{rem}
\section{ Structural results}
\vspace{5mm}
Recall the following relations between $D_{n,m}$ and $D_{n,p}$ (see \cite{abook}). Let $Q$ be any property of $D_{n,m}$ for some $m$, $0 \leq m \leq n(n-1)$ and let $p=\frac{m}{n(n-1)}$ then,
\begin{equation}
\mathbb{P}(D_{n,m} \text{ has } Q) \leq 10 \sqrt{m} \P(D_{n,p} \text{ has } Q).
\end{equation}
Moreover if $Q$ is a monotone increasing property i.e.\@ it is preserved under edge addition or monotone decreasing property i.e.\@ it is preserved under edge deletion, then we have
\begin{equation}
\mathbb{P}(D_{n,m} \text{ has } Q) \leq 3\mathbb{P}(D_{n,p} \text{ has } Q).
\end{equation}
For $p\in [0,1]$ we denote by $Bin (k,p)$ the random variable following the Binomial distribution with $k$ objects each appearing with probability $p$. Also, we will make use of the Chernoff bounds (see \cite{JLR}): namely, if $X$ is a $Bin(k,p)$ random variable with mean $\mu=np$ then for any $\epsilon>0$ we have
\begin{align}
Pr[X \leq (1-\epsilon)\mu] \leq e^{-\frac{\epsilon^2\mu}{2}}, \\
Pr[X \geq (1+\epsilon)\mu] \leq e^{-\frac{\epsilon^2\mu}{2+\epsilon}}.
\end{align}
Finally for the rest of the paper we let
$$ p_\ell=\frac{\log n+(q-1)\log \log n-\omega(n)}{n}, \hspace{10mm}  \hspace{10mm} m_\ell=n(n-1)p_\ell,  $$
and $$ p_u=\frac{\log n+(q-1)\log \log n+\omega(n)}{n}, \hspace{20mm} m_u=n(n-1)p_u,$$
where $\omega (n) = \frac{1}{2}\log \log \log n$. Recall that w.h.p.\@ $D_{n,m_\ell}$ has zero vertices of in- or out- degree less than $q-1$. In addition w.h.p.\@ $m_\ell \leq \tau \leq m_u$.

\begin{lem}\label{lem1}
W.h.p.\ for $k \in \big[q-1, \frac{3\log n}{\log \log n}\big],$  $D_{n,m_\ell}$ has at most $v_k:=\frac{e^{2\omega(n)}(\log n)^{k-q+1}}{(k-1)!} $ vertices of in-degree \emph{at most} $k$. Hence, the same is true for vertices of in-degree \emph{exactly} $k$, and similarly for out-degree $k$.
\end{lem}
\begin{proof}
By taking a union bound and using (2) for the first inequality, we get
\begin{align*}
\mathbb{P}( &D_{n,m_\ell} \text{ has more than } v_k \text{ vertices of in-degree at most } k )\\
& \leq \binom{n}{v_k}  \Bigg[3 \sum_{j=0}^{j=k} \binom{n-1}{j}  (1-p_\ell)^{n-j-1} p_\ell^j\Bigg]^{v_k}\leq  \bigg(\frac{en}{v_k}\bigg)^{v_k} \bigg[ 3(k+1)\binom{n-1}{k} (1-p_\ell)^{n-k-1}p_\ell^k \bigg]^{v_k}\\
&\leq \Bigg[\frac{en}{v_k} \frac{3(k+1)n^k}{k!}e^{-\log n -(q-1)\log \log n +\omega(n)+o(1)} \bigg(\frac{\log n +(q-1) \log n\log n -\omega(n)}{n}\bigg)^k\Bigg]^{v_k}\\
&\leq \Bigg[e^{ -\omega(n) +O(1)} \bigg(1+\frac{q\log \log n}{\log n}\bigg)^k\Bigg]^{v_k}
\leq \bigg[e^{ -\omega(n) +O(1)+\frac{ q\log \log n}{\log n} k}\bigg]^{v_k} \leq e^{ -\frac{\omega(n)v_k}{2}}.
\end{align*}
Hence
\begin{align*}
\mathbb{P}\bigg( \text{for some } k &\in\bigg[q-1, \frac{3\log n}{\log \log n}\bigg] \text{ there are more than } v_k \text{ vertices of out-degree k in } D_{n,m_\ell}\bigg)
\\ & \leq \sum_{k=q-1}^{\frac{3\log n}{\log \log n}} e^{ -\frac{\omega(n)v_k}{2}}=\sum_{k=q-1}^{\frac{3\log n}{\log \log n}} \big(e^{ -\frac{1}{4} \log \log \log n}\big)^{v_k} =o(1).
\qedhere
\end{align*}
\end{proof}
\begin{defn}\label{undd}
 For $u,v \in V_n$ let the undirected distance from $u$ to $v$ at time $t$, denoted by $d'_t(u,v)$, be the distance from $u$ to $v$ in the graph that is obtained from $D_t$ when we ignore the orientations of the edges.
\end{defn}
\begin{defn}\label{SMALL}
Let $SMALL:=\{ v \in V: d^+_{\tau}(v)\leq \frac{\log n}{100} \text{ or } d^-_{\tau}(v) \leq \frac{\log n}{100} \}$. Since we expect $\tau \geq n \log n$, $SMALL$ consists of vertices with significantly smaller degree than their expected value. 
\end{defn}
\begin{lem}\label{lem2}
 W.h.p.\ for every $v,w \in $ $SMALL,
  d_{\tau}'(v,w) \geq 2$.
\end{lem}
\begin{proof}
We weaker the definition of $SMALL$ so that it suffices to do the computation in $D_{m_u}$. Specifically, 
set $SMALL'$:=$\{ v \in V: d^+_{m_u}(v)\text{ or } d^-_{m_u}(v) \leq \frac{1}{100}\log n+ 2\omega(n) \}$. (1) gives us
\begin{align*}
\mathbb{P}\big(   v,w &\in SMALL' \text{ and } d_{m_u}'(v,w) \leq 2 )\\
& \leq 10 \sqrt{m_u} \sum_{k=1,2} \binom{n-2}{k-1}(2p_u-p_u^2)^k \Bigg[2 \mathbb{P} \bigg(Bin( n-1-k,p_u) \leq \frac{\log n}{100}+2\omega(n) -1    \bigg)\Bigg]^2\\
&\leq  \frac{200\sqrt{n} {\log }^{2.5}n}{n}  \Bigg[\exp \bigg(-(1-o(1))\log n \bigg(\frac{1}{100} \log \frac{1}{100} +\frac{99}{100}  \bigg) \bigg)\bigg]^2  =o(n^{-2.3}).
\end{align*}
 At the second inequality we used that $\mathbb{P}(Bin(\lambda / p,p)\leq \lambda-t) 
\leq \exp \{-\lambda[(1+x)\log (1+x)-x]\}$ (see \cite{JLR}), with $x=-\frac{t}{\lambda} \sim -\frac{99}{100}$ for $\lambda=(n-1-k)p_u \sim \log n, t=\lambda - \frac{\log n}{100}-2 \omega(n)$ here.
In the event $m_\ell \leq \tau \leq m_u$, as $D_\tau $ precedes $D_{m_u}$, we have that  $E_\tau \subseteq E_{m_u}$ and $|E_{m_u}\backslash E_\tau| \leq 2\omega(n)$. Furthermore if $d'_{\tau}(v,w) \leq 2$ then $d'_{m_u}(v,w) \leq 2$. Therefore $m_\ell \leq \tau \leq m_u$ implies that $SMALL\subseteq  SMALL'$. Hence,
\begin{equation*}
\mathbb{P} \Big( \exists v,w \in  SMALL \text{ such that } d_{\tau}' (v,w) \leq 2\Big) \leq \binom{n}{2} o(n^{-2.3})+\P\Big( \tau \notin [m_\ell, m_u]\Big) =o(1).
\qedhere
\end{equation*}
\end{proof}
\begin{notn}
For a digraph $D$ denote by $\Delta^+(D)$ and $\Delta^-(D)$ its maximum out- and in-degree respectively. Futhermore set 
$\Delta(D)=\max \{ \Delta^+(D), \Delta^-(D) \}$.
\end{notn}
\begin{lem}\label{maxdegree}
W.h.p.\@ $\Delta(D_\tau)\leq 12\log n$.
\end{lem}
\begin{proof}
We implicitly condition on the event $\{\tau\leq m_u\}$. Using (2)
\begin{align*}
\mathbb{P}\bigg( \Delta^+(D_\tau) \text{ or } \Delta^-(D_{\tau}) \geq 12\log n \bigg)
& \leq  3 \cdot2n \binom{n-1}{12\log n} p_u^{12\log n}
\leq 6 n \bigg( \frac{en}{12\log n}\bigg)^{12\log n} p_u^{12\log n} \\
& \leq 6 n \bigg( \frac{en}{12\log n} \cdot \frac{2\log n}{n} \bigg)^{12\log n} 
=o(1).
\qedhere
\end{align*}
\end{proof}

\begin{lem}\label{lem4}
W.h.p.\@ $\Delta(D_{m_1}) \leq \frac{\log n}{10^3q}$.
\end{lem}
\begin{proof}
 Recall $p_1=\frac{m_1}{n(n-1)}=\frac{e^{-q\cdot10^4}\log n}{ n-1}$. Then (2) gives us that
\begin{align*}
\mathbb{P}\bigg( \Delta^+(D_{m_1})\text{ or } \Delta^-(D_{m_1}) \geq  \frac{\log n}{10^3q} \bigg)
& \leq  3 \cdot2n \binom{n-1}{\frac{\log n}{10^3q}} p_1^{\frac{\log n}{10^3q}}
\leq 6 n \bigg( \frac{10^3qe(n-1)}{\log n}\bigg)^{\frac{\log n}{10^3q}} p_1^{\frac{\log n}{10^3q}} \\
& \leq 6 n \bigg( 10^3qe^{-q\cdot10^4+1}\bigg)^{\frac{\log n}{10^3q}}
=o(1).
\qedhere
\end{align*}
\end{proof}

\section{Minimum degree 1 in color $c$}
\begin{tm}\label{tm2}
W.h.p.\ COL succeeds in assigning colors to the arcs so that $\forall c \in [q]$ and $\forall v \in V_n$ we have
$d_{\tau}^+(v,c),$ $d^-_{\tau}(v,c) \geq 1$.
\end{tm}
We will approach this theorem by conditioning on the final digraph $D_\tau$ (in particular, on Lemmas \ref{lem1} and \ref{lem2}) and analysing the randomness of the edges' order and color. 
By symmetry, it suffices to prove the out-degree part. The proof will follow from Lemmas \ref{lem12}, \ref{lem13} given below.
\vspace{5mm}
\\For most of this section, at least until Lemma \ref{lem13}, $v \in V_n$ will be arbitrary but (crucially) fixed. 
Denote by $N^+(v)$ the out-neighbours of $v$ in $D_{\tau}$ and set $N_L^+(v):=N^+(v) \backslash SMALL_{\tau}$-  we aim for these larger neighbours to provide $v$ with the colors it needs, and thankfully, Lemma \ref{lem2} ensures $\leq 1$ neighbour was in $SMALL$.
 Furthermore let
 $A^+_L(v)$ be the set of arcs arising from $N_L^+(v)$ $\big($i.e.\
 $A_L^+(v)$:=
 $\{vw \in E_{\tau}:w \in N_L^+(v) \}\big)$.
  For $w \in N_L^+(v)$ we fix a set $B^-_{v}(w)$ of  $\frac{\log n}{100}-1$ arcs
  in $\big(V_n\backslash\{v,w\}\big)\times \{w\}$.
 Finally we let $A^-_v(w):= B^-_v(w) \cup \{vw\}$.
\begin{figure}[!h]
\centering
\begin{tikzpicture}[xscale=1]
\draw[fill] (-1.5,1) circle [radius=0.04]; \node at (-1.4,0.8) {$v$};
\draw (3,0.9) ellipse (0.3cm and 1cm);   \node at (3.7,0) {$N_L^+(v)$};
\draw[fill] (3,2.1) circle [radius=0.03]; \draw[->] (-1.5,1)--(2.95,2.08);
\draw[fill] (3,0.4) circle [radius=0.03]; \draw [->,blue] (-1.5,1)--(2.95,0.41);
\draw[fill] (3,0.6) circle [radius=0.03]; \draw [->,blue] (-1.5,1)--(2.95,0.61);
\draw[fill] (3,1) circle [radius=0.03]; \draw [->,blue] (-1.5,1)--(2.95,1);
\draw[fill,] (3,1.6) circle [radius=0.03]; \draw [->,blue] (-1.5,1)--(2.95,1.58);  \node at (3.15,1.4) {$w$};
\draw [dotted,thick] (3,1.5) -- (3,1.1);
\draw [dotted,thick] (3,0.9)--(3,0.7);
\draw [dotted,thick] (1.25,1.35) -- (1.25,1);
\draw [dotted,thick] (1.25,1)--(1.25,0.6);
\node[below] at (1.25,0.6) {$A_L^+(v)$};
\draw (6.5,1.4) ellipse (0.25cm and 0.6cm);   \node at (3.7,0) {$N_L^+(v)$};
\draw[fill] (6.5,1.8) circle [radius=0.03]; \draw [->,red] (6.5,1.8)--(3.05,1.61);
\draw[fill] (6.5,1) circle [radius=0.03]; \draw [->,red] (6.5,1)--(3.05,1.58);
\draw [dotted,thick] (6.5,1)--(6.5,1.8);  \node[below] at (5,1.25) {$B_v^-(w)$};
\draw [dotted,thick] (5,1.7)--(5,1.25);
\end{tikzpicture}
\caption{arcs in $A_L^+(v)$ and in $B_u^-(w)$ are in blue and red respectively.}
\end{figure}
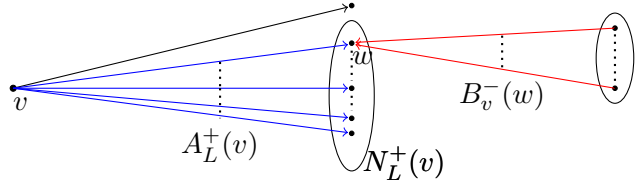
\\
We will only need to analyse the algorithm's effect on $\bigcup_w A_v^-(w)$ to show $v$ is unlikely to obtain all the colors it needs.
For this analysis, we couple the algorithm as follows.
Let $D_{\tau}^{(1)}$ and $D_{\tau}^{(2)}$ be two copies of $D_{\tau}$ colored in parallel according to algorithm $COL1(v)$ given below.
$D_\tau^{(2)}$ will mimic $COL$. Meanwhile, $D_\tau^{(1)}$ will be strictly worse (for $v$'s satisfaction), but will color $\bigcup_w B_v^-(w)$ fully randomly, and thus will be easier to analyse.
\begin{notn}
For $i \in [2]$ we extend the notation $C^+_v(t)$, $C^-_v(t), FULL^+_t,$ $FULL^-_t$, $BAD$  to $C^+_{i,v}(t)$, $C^-_{i,v}(t)$, $FULL^+_{i,t}$, $FULL^-_{i,t}$ and $BAD_{i}$
for the corresponding sets in  $D_{\tau}^{(i)}$.
\end{notn}
\begin{algorithm}[H]
\caption{COL1($v$)}
\For{ $t=1,...,\tau$}{
    let $e_t=xy$ \\
   \eIf{ $e_t \in \underset{w \in N_L^+(v)}{\bigcup}B^-_v(w) $}{
   choose a color $c$ from $[q]$ uniformly at random \\
        \eIf{ $c \in C^+_{2,x}(t-1) \cup C^-_{2,y}(t-1)$ }{ color $e_t$ in both $D_{\tau}^{(1)}$, $D_{\tau}^{(2)}$ with color $c$.}
        { color $e_t$ in  $D_{\tau}^{(1)}$  with color $c$, \\
          to color $e_t$ in $D_{\tau}^{(2)}$ execute step $t$ of $COL$, \footnotemark
          }
     }
     {to color $e_t$ in $D_{\tau}^{(2)}$ execute step $t$ of $COL$. \footnotemark[\value{footnote}] \\
      color $e_t$ in $D_{\tau}^{(1)}$ by the same color as in $D_{\tau}^{(2)}$.
}}
\end{algorithm}
\footnotetext{ Here we suppose that we run $COL$. Our current arcs $e_1,...,e_{(t-1)}$ have the colors that have been  assigned by $COL1(v)$ to the corresponding arcs  in $D_\tau^{(2)}$. We use 
 $FULL^+_{2,t}$, $FULL^-_{2,t}$ and $BAD_{2}$ in place of $FULL^+_t,$ $FULL^-_t$ and $BAD$ respectively. 
}
\begin{rem}\label{rr}
 The colorings of $D_\tau^{(2)}$ and $D_\tau$ have the same distribution.
\end{rem}
\begin{rem}\label{r12}
For every $t \in [\tau]$ and $w \in N_L^+(v)$ since the algorithm may color an arc $e_{t}=xw$ in $D_\tau^{(1)}$ and in $D_\tau^{(2)}$ with distinct colors $c$ and $c'$ respectively only in the case where $c \notin C^+_{2,x}(t-1) \cup C^-_{2,w}(t-1)$ (i.e $c\notin C^-_{2,w}(t-1)$) we have $C^-_{2,w}(t)\subseteq C^-_{1,w}(t)$.
\end{rem}
\begin{defn}\label{contr}
Fot $t\in [\tau]$ we say that $e_t \in A^+(v)$ contributes to the coloring of $v$ (or just contributes to $v$) in $D_{\tau}^{(1)}$ if
\emph{either} $C^+_{1,v}(t-1)= \emptyset$  \emph{or} $e_t$ gets a color in $C^+_{1,v}(t-1)$.
\end{defn}
\begin{lem}\label{lem11}
 Once $q$ arcs have contributed to the coloring of $v$ in $D_{\tau}^{(1)}$ we have that in $D_{\tau}^{(2)}$, $v$ has out-degree at least one in each color.
\end{lem}
\begin{proof}
Follows directly from Definition \ref{contr} and Remark \ref{r12}.
\end{proof}
The strength of Lemma \ref{lem11} is that it allows us to do the desired computations in $D_\tau^{(1)}$, for Lemmas \ref{lem12} and \ref{lem13}. 
\begin{lem}\label{lem12} 
For any $v\in V_n$, if we run the corresponding coloring algorithm $D_\tau^{(1)}$:
$$ \mathbb{P}\big( \text{less than } q \text{ arcs contribute to the coloring of $v$ in $D_{m_\ell}^{(1)}$}\big) \leq \binom{d^+(v)-1}{q-1}\bigg(\frac{100q^{q+1}}{\log n}\bigg)^{d^+(v)-q}.$$
\end{lem}
Before proceeding to the proof of Lemma \ref{lem12} we introduce the following two functions.
\begin{defn}
For $e\in E_\tau$ define the bijection $h:E_\tau\rightarrow [\tau]$ where $h(e)=k$ means $e=e_k$, i.e\@ $e$ was the $k$th arc to be revealed. Thus, for example, $FULL^-_{1,h(vw)}=FULL^-_{1,t'}$ where
$e_{t'}=vw$.
\end{defn}
\begin{defn}
For $w\in N_L^+(v)$ define the bijection $g_{v,w}:A^-_v(w) \rightarrow \big[\frac{\log n}{100}\big]$ where $g_{v,w}(xw)=k$ means $xw$ is the $k$th arc that was revealed out of all the arcs in $A^-_v(w)$.
\end{defn}
Also we define the following events.
\begin{defn}
For $w \in N^+_L(v)$ set $F(w)$  to be the event that in $D_{\tau}^{(1)}$ $\nexists \ell \in \mathbb{Z}_{\geq 0} $ s.t. $\ell q+q< g_{v,w}(vw)$ and  $g_{v,w}^{-1}(\ell q+1),...,g_{v,w}^{-1}(\ell q+q)$  are colored by $q$ distinct colors.
\end{defn}
\begin{rem}\label{inc}
For every $w \in N_L^+(v)$,
the event $\{ w \notin FULL_{1,h(vw)}^- \} \subseteq F(w)$.
\end{rem}
Indeed, for any $\ell \in \mathbb{Z}_{\geq 0} $ such that $\ell q+q< g_{v,w}(vw)$ the arcs $g_{v,w}^{-1}(\ell q+1),...,g_{v,w}^{-1}(\ell q+q)$ precede $vw$.
So if they were colored differently,
we would have $w \in FULL_{1,h(vw)}^-$, which is the contrapositive.
\begin{rem}
The events $\{F(w):w \in N_L^+(v)\}$ are independent.
\end{rem}
Indeed, for $w \in N^+_L(v), \P(F(w))$ depends only on the relative time $g_{v,w}(vw)$ of $vw$ among in-edges of $w$. That is because the colors that \emph{COL1(v)} assigns to the edges, $g_{v,w}^{-1}(1),g_{v,w}^{-1}(2),....,$ $g_{v,w}^{-1}\big(g_{v,w}(vw)-1\big)$, preceding $vw$ are chosen independently and uniformly at random from $[q]$. Thus in showing the independence of $\{F(w)\}$ it suffices to note that the values $\{ g_{v,w}(vw): w \in N^+_L(v)\}$ are independent, and this follows from the sets $A_v^-(w)$ being disjoint.
\vspace{3mm}
\\ \emph{Proof of Lemma} \ref{lem12}:
For $w \in N^+_L(v)$,
\begin{align}\label{above}
\begin{split}
\mathbb{P}\big( F(w)\big)
&= \sum_{k=1}^{\frac{\log n}{100}} \mathbb{P}\big( \{g_{v,w}(vw)=k\} \wedge  F(w)  \big)
= \sum_{k=1}^{\frac{\log n}{100}}\mathbb{P}\big( g_{v,w}(vw)=k\big)\mathbb{P}\big( F(w)\vert g_{v,w}(vw)=k\big)\\
& \leq \sum_{k=1}^{\frac{\log n}{100}} \frac{100}{\log n}
\prod_{l=1}^{\lfloor{ (k / q)-1} \rfloor} \left(1-\frac{1}{q^q}\right)
\leq\frac{100}{\log n} \sum_{j \in \mathbb{Z}_{\geq 0}} q \bigg(1-\frac{1}{q^q}\bigg)^j \leq \frac{100 }{\log n}q^{q+1}.
\end{split}
\end{align}
Hence,
\begin{align*}
\mathbb{P}\big(\text{less }&\text{than } q \text{ arcs contribute to the coloring of $v$ in $D_{\tau}^{(1)}$  } \big) \\
& \leq \mathbb{P}\bigg( \left\vert \left\{w \in N_L^+(v): w  \notin FULL^-_{1,g_{v,w}^{-1}(vw)}\right\}\right\vert \geq  d^+(v)-q  \bigg) \\
& \leq \mathbb{P}\bigg( \left\vert \left\{w \in N_L^+(v): \text{event }F(w)\text{ occurs }\right\}\right\vert \geq  d^+(v)-q  \bigg) \\
&\leq \mathbb{P} \Bigg(Bin \bigg(d^+(v)-1,
\frac{100q^{q+1}}{\log n}
\bigg) \geq d^+(v)-q
\Bigg)
\leq \binom{d^+(v)-1}{q-1}\bigg(\frac{100q^{q+1}}{\log n}\bigg)^{d^+(v)-q}.
\qed
\end{align*}
 The second inequality follows from Remark \ref{inc}. The last inequality follows from   the independence of  the events $\{F(w)\}$, the fact that  $\left\vert N_L^+(v) \right\vert \geq d^+(v)-1$ (see Lemma \ref{lem2}) and (\ref{above}).
\begin{rem}\label{sss}
The two basic ingredients that are used in the proof of Lemma \ref{lem12} as well as in Lemma \ref{lem13} are the following: First, for $w\in N_L^+(v)$ the sets $B_v^-(w)$ are disjoint and of size $\Omega(\log n)$.
Second, in $D_{\tau}^{(1)}$ for every $w\in N_L^+(v)$ the arcs in $B_v^-(w)$  are colored independently and uniformly at random.
 The disjointness of the sets $B^-_v(w)$ implied the independence of the events $F(w)$ while the fact their size is $\Omega(\log n)$ leads to the desired probability being sufficiently small.

\end{rem}
The following remark will be used later in the proof of Lemma \ref{full}:

\begin{rem}\label{ss}
We could  reproduce the above lemma with different parameters and similar definitions. That is we could use 
$m_1$ in place of $\tau$, $N^+_{m_1}(v)$ to be the neighbours of $v$ in $D_{m_1}$ and for $w\in N^+_{m_1}(v)$  $B^-_{m_1,v}(w)$ to be a set of arcs in $E_{m_1}$ from $V_n\backslash\{v,w\}$ to $w$ of size $\gamma \log n$ where $\gamma$ is some positive constant. In this case for every $v\in V_c$ such that the condition $|\{w \in V_n: w\in N^+(v), h(vw)<m_1 \text { and } d^+_{m_1}(w) \leq \gamma \log n\}| \leq k$ (in place of Lemma \ref{lem2}) holds, using the same methodology, we could prove that
$$\P\bigl(\text{less than $q$ arcs contribute to $v$ in } D_{m_1}^{(1)}(v)\bigl) \leq \binom{d^+_{m_1}(v)-k}{q-1}\bigg(\frac{q^{q+1}}{\gamma \log n}\bigg)^{\left(d^+_{m_1}(v)-k\right)-(q-1)}.   $$
Hence, setting $d=\min\{d^+_{m_1}(v), d^-_{m_1}(v)\}$, we have
\begin{align*}
\P\big(  v \notin FULL_{m_1}^+ \cap FULL_{m_1}^-\big)
&\leq 2  \binom{d-k}{q-1}\bigg(\frac{q^{q+1}}{\gamma \log n}\bigg)^{(d-k)-(q-1)}.
\end{align*}
\end{rem}
The bound provided by Lemma \ref{lem12} is not strong enough for vertices of small out-degree. However, it can be improved by considering some extra information,  provided by Lemma \ref{lem13}. Suppose $e_{\tau_q}=(v^*,w^*)$. Since $e_{\tau_q}$ is the last arc of our process we have that either $d^+(v^*)=q$ or $d^-(w^*)=q$. In the case that $d^+(v^*)=q$ we handle $v^*$ separately. Otherwise $d^-(w^*)=q$ and Lemma \ref{lem2} implies that 
$d^+(v^*)>\frac{\log n}{100}.$ We may assume that  $d^+(v^*)=q$ and we deal with $v^*$ separately later.
\begin{lem}\label{lem13}
Let $v\in V_n\setminus v^*$ satisfy $q \leq d^+(v) \leq  \log \log n $. Then
the probability that fewer than $ q $ arcs contribute to the coloring of $v$ in $D_\tau^{(1)}$  is bounded above by $$\frac{101(\log \log n)^5}{\log n}\binom{d^+(v)-1}{ q-2} \bigg(\frac{101q^{q+1} }{\log n}\bigg)^{d^+(v)-q+1}.$$
\end{lem} 
In addition to the $\{g_{v,w}\}$ keeping track of the (random) relative timings of edges within each $A_v^-(w)$, we also care about the relative timings of edges within our entire subgraph $\bigcup_w A_v^-(w)$ and also within our most crucial edges $A_L^+(v)$ that we hope will contribute to $v$. We define the following two functions accordingly:
\begin{defn}
For each $v \in V_n$, let $g_{v}:A_L^+(v) \rightarrow \big[|A^+_L(v)|\big]$  map $vw \mapsto k$  whenever $vw$ is the $k$th arc revealed among $A^+_L(v)$.
Similarly define $h_v:\underset{w \in N_L^+(v)}{\bigcup}A^-_v(w) \rightarrow \left[ \frac{\log n}{100}\cdot|A^+_L(v)|\right]$.
\end{defn}
Observe that the maps $h_v(\cdot),g_v(\cdot)$ are also bijections.
\vspace{6mm}
\\ \emph{Proof of Lemma \ref{lem12}:}
Our strategy is as thus. Most of the time, we expect that none of the crucial edges in $A_L^+(v)$ appear before some time $r \ll \frac{\log n}{100}$, by which point we also expect that all $w \in N_L^+(v)$ have received a reasonable collection $1 \ll r_\ell \ll r$ of their own edges from \emph{other} vertices. It is unlikely that either of these heuristics fail (see bounds on $\mathbb{P}(A)$ and $\mathbb{P}(B)$ in Cases 1 and 2 below), and when they are correct (Case 3), all the $w$'s become measurably more likely to have become $FULL$ by the time edge $vw$ appears.
Specifically, with $r_\ell= q^{q}\log \log n$ and $r=(\log \log n)^5$ we define the events $A$ and $B$:
\begin{itemize}
\item{Let $A$  be the event $\{h_v(g_v^{-1}(1)) \leq r \}$;
 i.e.\@ the first arc of $A_L^+(v)$ precedes the $(r+1)$st of $\underset{w \in N_L^+(v)}{\bigcup}A^-_v(w)$.}
\item { Let $B$ be the event $\{
\exists w \in N_L^+(v):
 h_v(g_{v,w}^{-1}(r_\ell))>r+1 
\}$; i.e.\@ for some $w \in N_L^+(v)$, less than $r_\ell$ arcs in $A^-_v(w)$ are revealed before the $(r+1)$st arc of $\underset{w \in N_L^+(v)}{\bigcup}A^-_v(w)$.} 
\end{itemize}
We condition on whether $A$, $A^c \cap B$, or $A^c \cap B^c$ occurs. In each case we use the same methodology as in Lemma \ref{lem12} to bound the desired probability.
Observe  that Lemma \ref{lem2} implies, as $d_{\tau}^+(v) \leq  \log \log n $, that $v$ has no out-neighbour in $SMALL_{\tau}$, hence $N^+(v)=N^+_L(v)$. Furthermore note that in any of the events $A$, $A^c \cap B$ and $A^c \cap B^c$ the first arc that appears with out-vertex $v$ contributes to the colouring of $v$. Since $N^+(v)=N^+_L(v)$ that arc belongs to $A_L^+(v)$.
\vspace{3mm}
\\
\textbullet \hspace{1mm} Case 1: 
$A$ occurs.
We describe the possible offending sequences leading up to the early first edge of $A_L^+(v)$ as follows.
\vspace{3mm}
\\Set $\mathcal{E}_1=\Bigg\{(f_1,...f_s)\in \Bigg(\underset{w\in A_L^+(v)}{\bigcup} B^-_v(w)\Bigg)^{s-1}\times A_L^+(v) : s\leq r\text{ and }  f_1,...,f_s \text{ are distinct}\Bigg\}.$
\\For $E=(f_1,...f_s) \in \mathcal{E}_1$ we set $f_E:=f_s$ and define $A_E$ to be the event where both: 
\begin{itemize}
\item $f_s$ is the first arc to be revealed from $A_L(v)$, and
\item $f_1,...,f_{s-1}$ are the only arcs in  $\underset{w\in A_L^+(v)}{\bigcup} B^-_v(w)$ to be revealed before $f_s$.
\end{itemize}
Consequently the events $A_E$ partition $A$.
 We furthermore define the set $A^-_{v,E}(w)$, the function $g_{v,w,E}(vw)$ and the event $F(w,E)$ as follows.  We  set  $A_{v,E}^-(w)$ to be a subset of $A_{v}^-(w)\backslash E$ of size $\frac{\log n}{100}-r$ and we define  the map $g_{v,w,E}:A^-_{v,E}(w) \rightarrow \big[\frac{\log n}{100}-r\big]$ given by the relation $g_{v,w,E}(xw)=k$ where $xw$ is the $k$th arc that was revealed out of the arcs in $A^-_{v,E}(w)$.
In addition we set $F(w,E)$  to be the event that $A_E$ occurs and that in $D_{\tau}^{(1)}$ $\nexists \ell \in \mathbb{Z}_{\geq 0} $ s.t. $\ell q+q< g_{v,w,E}(vw)$ and  $g_{v,w,E}^{-1}(\ell q+1),...,g_{v,w,E}^{-1}(\ell q+q)$  are colored by $q$ distinct colors.
\\For $E \in \mathcal{E}_1$, suppose we condition on $A_E$. By using the same tools as in Lemma \ref{lem12} with $A^-_{v,E}(\cdot)$, $g_{v,\cdot,E}(v\cdot)$ and $F(\cdot,E)$ in  place of $A_{v}(\cdot)$, $g_{v,\cdot}(v\cdot)$ and $F(\cdot)$ respectively, we have that for $w \in N_L^+(v)\backslash\{v^*\}$ where $vv^*=f_E$ the events $F(w,E)$ occur  independently  with probability at most $\frac{q^{q+1} }{\frac{\log n}{100}-r}$.
On the other hand $f_E$ contributes to the the coloring of $v$ with probability 1. Therefore, as the events $A_E$ partition $A$, the probability that fewer than $q$ arcs contribute to the coloring of $v$ in $D_{\tau}^{(1)}$ \emph{conditioned on the event $A$} is bounded above  by
$$\mathbb{P}\Bigg(Bin\bigg(d^+(v)-1, \frac{q^{q+1} }{\frac{\log n}{100}-r}\bigg) \geq [d^+(v)-1]-(q-2)\Bigg).$$
As $\mathbb{E}\big[|A_L^+(v) \cap \{h_v^{-1}(1), h_v^{-1}(2),..., h_v^{-1}(r)\}|\big] = \frac{r}{|A_L^+(v)|\frac{\log n}{100}} \cdot |A_L^+(v)|$, Markov's inequality gives
$$\mathbb{P}(A) = \mathbb{P}\big(|A_L^+(v) \cap \{h_v^{-1}(1), h_v^{-1}(2),..., h_v^{-1}(r)\}| \geq 1\big) \leq \frac{100r}{\log n}.$$
\\\textbullet \hspace{1mm} Case 2: The event $A^c\cap B$ occurs.
\\Set $\mathcal{E}_2=\Bigg\{(f_1,...f_r)\in \Bigg(\underset{w\in A_L^+(v)}{\bigcup} B^-_v(w)\Bigg)^r: f_1,...,f_r \text{ are distinct and  } |\{f_1,...,f_r\}\cap A_v^-(w)|< r_\ell$ for some $ w \in N_L^+(v) \Bigg\}$. Henceforth we can proceed as in Case 1 but without using  the guaranteed contribution of the first arc in $A^+_L(v)$.
Thus, conditioned on the event $A^c\cap B$, the probability that fewer than $q$ arcs contribute to the coloring of $v$ in $D_{\tau}^{(1)}$ is bounded above  by
$$\mathbb{P}\Bigg(Bin\bigg(d^+(v), \frac{q^{q+1} }{\frac{\log n}{100}-r}\bigg) \geq d^+(v)-(q-1)\Bigg).$$
Furthermore,
\begin{align*}
\mathbb{P}(A^c \cap B) &\leq \mathbb{P}(B) \leq d^+(v)\sum_{i=0}^{r_\ell-1} \binom{\frac{\log n}{100}}{i} \binom{(d^+(v)-1)\frac{\log n}{100}}{r-i}\bigg/  \binom{d^+(v)\frac{\log n}{100}}{r} \\
&=d^+(v)\sum_{i=0}^{r_\ell-1} \binom{\frac{\log n}{100}}{i}  \binom{(d^+(v)-1)\frac{\log n}{100}}{r-i}
\binom{r}{ r-i}\bigg/  \binom{d^+(v)\frac{\log n}{100}}{r-i}  \binom{d^+(v)\frac{\log n}{100}-r+i}{i}\\
&\leq d^+(v)\sum_{i=0}^{r_\ell-1} r^i  \binom{(d^+(v)-1)\frac{\log n}{100}}{r-i}   \bigg/  \binom{d^+(v)\frac{\log n}{100}}{r-i}\\
&\leq d^+(v)\sum_{i=0}^{r_\ell-1} r^i  \prod_{j=0}^{r-i-1}\frac{(d^+(v)-1)\frac{\log n}{100}-j}{ d^+(v)\frac{\log n}{100}-j}
\leq d^+(v)\sum_{i=0}^{r_\ell-1} r^i\bigg(\frac{(d^+(v)-1)\frac{\log n}{100}}{ d^+(v)\frac{\log n}{100}}\bigg)^{r-r_l}\\
& \leq d^+(v)\cdot r^{r_\ell} \cdot \exp\bigg\{-\frac{r-r_l}{d^+(v)}\bigg\}\\
&\leq \exp\bigg\{ \log\big(d^+(v)\big)+q^{q}\log\log n\cdot 5\log(\log\log n) -0.4(\log\log n)^4 \bigg\}=o\bigg(\frac{1}{\log ^3n}\bigg).
\end{align*}
To get from the second to the third line we are using the fact that $d^+(v)\geq 2$. Furthermore at the last inequality we use that  $d^+(v) \leq \log \log n$.
\vspace{3mm}
\\\textbullet \hspace{1mm} Case 3: The event $A^c\cap B^c$ occurs.
\vspace{3mm}
\\Set $\mathcal{E}_3=\Bigg\{(f_1,...f_r) \in \Bigg(\underset{w\in A_L^+(v)}{\bigcup} B^-_v(w)\Bigg)^r: f_1,...,f_r$ are distinct and for every   $ w \in N_L^+(v)$ we have that $|\left\{f_1,...,f_r\right\}\cap B_v^-(w)|\geq r_\ell \Bigg\}$.
For $E \in \mathcal{E}_3$ we let
 $A_E$  be the event that for all $i \in [r],$ $f_i $ is the $i$-th edge that is revealed from $\underset{w\in A_L^+(v)}{\bigcup} A^-_v(w)$. Consequently we have that the events $A_E$  partition the event $A^c\cap B^c$. 
 Furthermore for $E=(f_1,...f_r) \in \mathcal{E}_3$ and $w\in N^+_L(v)$ we set $\tilde{A}_{v,E}^-(w)$ to be a subset of $A_{v}^-(w)$ of size $\frac{\log n}{100}-r+r_\ell$ such that  $|\tilde{A}_{v,E}^-(w)\cap \{e_1,...,e_r\}|=r_\ell$ and define the map $\tilde{g}_{v,w,E}:\tilde{A}^-_{v,E}(w) \mapsto \big[\frac{\log n}{100}-r+r_\ell\big]$  and the event $\tilde{F}(w,E)$  correspondingly. Note that for $w \in N_L^+(v)$ and for $E \in \mathcal{E}_3$ since $A_E \subseteq A^c\cap B^c$ we have that  $\tilde{g}_{v,w,E}(vw)>r_\ell$. Thus, as in the proof of Lemma \ref{lem12} for any $E \in \mathcal{E}_3$ and  $w \in N^+_L(v)$ we have,
\begin{align*}
\mathbb{P}\big( \tilde{F}(w,E)\vert A_E \big)
&= \sum_{k=r_\ell+1}^{\frac{\log n}{100}-r+r_\ell} \mathbb{P}\big( \tilde{g}_{v,w,E}(vw)=k \wedge  \tilde{F}(w,E) \vert A_E  \big) \\
&= \sum_{k=r_\ell+1}^{\frac{\log n}{100}-r+r_\ell}\mathbb{P}\big( \tilde{g}_{v,w,E}(vw)=k\vert A_E \big) \mathbb{P}\bigg( \tilde{F}(w,E)\vert \tilde{g}_{v,w}(vw)=k \wedge A_E\bigg)\\
& \leq \sum_{k=r_\ell}^{\frac{\log n}{100}} \frac{1}{\frac{\log n}{100}-r} \bigg(1-\frac{1}{q^q}\bigg)^{\lfloor{ k / q} \rfloor}
 \leq  \sum_{j \in \mathbb{N}} \frac{100}{\log n-100r} \bigg(1-\frac{1}{q^q}\bigg)^{\lfloor r_\ell \rfloor} \bigg(1-\frac{1}{q^q}\bigg)^j \\
& \leq \sum_{j \in \mathbb{N}} \frac{101}{\log n} \cdot \exp\bigg({-\frac{1}{q^q}\cdot{\lfloor q^{q}\log \log n \rfloor}}\bigg) \cdot \bigg(1-\frac{1}{q^q}\bigg)^j  \leq  \frac{101eq^q }{\log^{2} n}.
\end{align*}
Once more, for fixed $E \in \mathcal{E}_3$, conditioned on $A_E$ the events $F(w,E)$ are independent (as in case 1). Furthermore the events $A_E$ for   $E \in \mathcal{E}_3$ partition $A^c \cap B^c$. Hence, conditioned on the occurrence of event $A^c \cap B^c $  the probability that less than $ q $ arcs contribute to the coloring of $v$ in $ D_\tau^{(1)}$ is bounded  by
$$P\Bigg(Bin\bigg(d^+(v), \frac{101eq^q}{\log^{2}n}\bigg) \geq d^+(v)-(q-1)\Bigg).$$
Finally, by conditioning on the occurrence of event $A$ or $A^c \cap B $ or $A^c \cap B^c$ we get that for a vertex $v$ in $D_{\tau}^{(1)}$  satisfying $q \leq d^+(v) \leq  \log \log n $ we have,
\begin{align*}
\mathbb{P}\big( &\text{fewer than } q \text{ arcs contribute to the coloring of $v$ in } D_{\tau}^{(1)}\big)\\
& \leq \mathbb{P}\Bigg(Bin\bigg(d^+(v)-1, \frac{100q^{q+1} }{\log n-100(\log \log n)^5}\bigg) \geq [d^+(v)-1]-(q-2)\Bigg) \frac{100(\log \log n)^5}{\log n}\\
&  + \mathbb{P}\Bigg(Bin\bigg(d^+(v), \frac{100q^{q+1} }{\log n-100(\log \log n)^5}\bigg) \geq d^+(v)-q+1\Bigg) \frac{1}{\log^3 n}\\
&  + \mathbb{P}\Bigg(Bin\bigg(d^+(v), \frac{101eq^q}{\log^{2}n}\bigg) \geq d^+(v)-q+1\Bigg) \\
& \leq \frac{101(\log \log n)^5}{\log n}\binom{d^+(v)-1}{ q-2 } \bigg(\frac{101q^{q+1} }{\log n}\bigg)^{d^+(v)-q+1}. \qed
\end{align*}

\begin{lem}\label{new}
Let $e_{t_q}=(v^*,w^*)$ be such that $d^+(v^*)=q$.
Then probability that fewer than $ q $ arcs contribute to the coloring of $v^*$ in $D_\tau^{(1)}$  is bounded above by 
$$\frac{101(\log \log n)^5}{\log n}\binom{d^+(v)-1}{ q-2} \bigg(\frac{101q^{q+1} }{\log n}\bigg)^{d^+(v)-q+1}.$$
\end{lem}
\begin{proof}
As seen in the proof of Lemma \ref{lem12} every arc out of $v^*$ except $(v^*w^*)$ contributes to the coloring of $v$ with probability  $=\frac{100q^{q+1}}{\log n}$. Thereafter since $g_{v,w}(v^*w^*)=\frac{\log n}{100}$ the first line of \eqref{above} gives as 
$$ \mathbb{P}(F(w^*))
\leq   \prod_{l=1}^{\lfloor{ \frac{\log n}{100} / q} \rfloor} \left(1-\frac{1}{q^q}\right) 
\leq  \left(1-\frac{1}{q^q}\right)^{\frac{\log n}{100q}} 
\leq e^{-\frac{\log n}{100q^{q+1}}}\leq  \frac{100q^{q+1}}{\log n}.
$$
Therefore the probability that fewer than $ q $ arcs contribute to the coloring of $v^*$ in $D_\tau^{(1)}$  is bounded above by $q\cdot \frac{100q^{q+1}}{\log n}\leq \frac{101(\log \log n)^5}{\log n}\binom{d^+(v)-1}{ q-2} \bigg(\frac{101q^{q+1} }{\log n}\bigg)^{d^+(v)-q+1}.$

\end{proof}
\vspace{3mm}
{\bf Proof of Theorem \ref{tm2}:}  We say  $COL$ fails if once the last edge has been revealed, there exist a vertex $v \in V$ and a color  $c\in [q]$ such that the in- or out-degree of $v$ in color $c$ is 0. Observed that conditioned on the almost sure event $\{m_\ell\leq \tau \}$ Lemma \ref{lem1} implies that for all $k\in[q,  {3\log n}\backslash {\log \log n}]$ the number of vertices of degree at most $k$ is at most $v_k = e^{2\omega(n)}(\log n)^{k-q+1}/ (k-1)!$.  Thus from Lemmas \ref{lem11}, \ref{lem12}, \ref{lem13} and Remark \ref{rr}, by implicitly conditioning on the event  $\{m_\ell\leq \tau \}$ and Lemma \ref{lem1},  we have
\begin{align*}
\mathbb{P}(&\text{COL}\text{ fails}) \leq 2 \mathbb{P}\big( \exists v \in D_{m_\ell} \text{ such that less than } q \text{ arcs contribute to the coloring of $v$ in } D_{m_\ell}^{(1)}\big) \\
&\leq 2\sum_{k=\frac{3\log n}{\log \log n }}^{n} n \cdot \binom{k-1}{q-1}\bigg(\frac{100q^{q+1}}{\log n}\bigg)^{k-q}
+2 \sum_{k=\log \log n +1 }^{\frac{3\log n}{\log \log n }}v_k\cdot  \binom{k-1}{q-1}\bigg(\frac{100q^{q+1}}{\log n}\bigg)^{k-q} \\
&+2\sum_{k=q }^{{\log \log n }} v_k \cdot \frac{101(\log \log n)^5}{\log n} \binom{k-1}{ q-2} \bigg(\frac{101q^{q+1} }{\log n}\bigg)^{k-q+1}\\
&\leq 2\sum_{k=\frac{3\log n}{\log \log n }}^{n} n\cdot {k}^q \bigg(\frac{100q^{q+1}}{\log n}\bigg)^{k-q}
+ 2\sum_{k=\log \log n +1 }^{\frac{3\log n}{\log \log n }}\frac{e^{2\omega(n)}(\log n)^{k-q+1}}{(k-1)!} \cdot    \binom{k-1}{q-1}\bigg(\frac{100q^{q+1}}{\log n}\bigg)^{k-q} \\
&+2\sum_{k=q }^{{\log \log n }}\frac{e^{2\omega(n)}(\log n)^{k-q+1}}{(k-1)!} \cdot \frac{101(\log \log n)^5}{\log n}q\binom{k-1}{ q-1} \bigg(\frac{101q^{q+1} }{\log n}\bigg)^{k-q+1} \\
&\leq 2\sum_{k=\frac{3\log n}{\log \log n }}^{n} \frac{1}{n^2}
+ 2\sum_{k=\log \log n +1 }^{\frac{3\log n}{\log \log n }}\frac{e^{2\omega(n)} \log n}{(k-q)!} (100q^{q+1})^{k-q} \\
&+2\frac{101^2q^{q+2}(\log \log n)^5\cdot e^{2\omega(n)}}{\log n}\Bigg[\sum_{k=q+1+202eq^{q+1}}^{\log\log n}\frac{(101q^{q+1})^{k-q}}{(k-q)!}+
\sum_{k=q }^{{q+202eq^{q+1} }}\frac{(101q^{q+1})^{k-q}}{(k-q)!}
\Bigg]\\
&\leq \frac{2}{n}+ 2\log^2 n\sum_{k=\log \log n +1 }^{\frac{3\log n}{\log \log n }}\bigg(\frac{100q^{q+1}e}{k-q}\bigg)^{k-q} \\
&+\frac{C_1(\log \log n)^6}{\log n}\Bigg[\sum_{k=q+1+202eq^{q+1}}^{\log\log n}\bigg(\frac{101q^{q+1}e}{(k-q)}\bigg)^{k-q}+
C_2   \Bigg] \\
&\leq \frac{2}{n}+
 2 \log^2 n \cdot \frac{3\log n}{\log \log n } \bigg(\frac{100q^{q+1}e}{\log \log n-q}\bigg)^{\log \log n-q}
+ O\bigg(\frac{(\log \log n)^6}{\log n}\bigg)=o(1),
\end{align*}
for some sufficiently large constants $C_1=C_1(q)$ and $C_2=C_2(q)$ depending only on $q$. \hfill\({\Box}\)

\section{Finding Hamilton cycles - Overview }
We may now proceed to show that w.h.p.\@ for every color $c \in [q]$, $COL$ succeeds in assigning color $c$ to every edge in some Hamilton cycle in $D_\tau$. We set $D_c'$ to be the subgraph of $D_\tau$ induced by the edges of color $c$. We start by constructing a minor  $D_c$ of $D_c'$. To do so we first remove some arcs and then applying contractions to arcs adjacent to vertices in $BAD$. By doing the contractions we  hide  the vertices in $BAD$ while the arc removal ensures that any Hamilton cycle in $D_c$ also yields a Hamilton cycle in $D_c'$.

We organize the rest of the proof as follows.
We first deal with \emph{Phase 1} which takes place in our original setting.
We then give a reduction of Theorem \ref{maintm} to Lemma \ref{reduction}. 
Finally we explicitly describe \emph{Phases 2 \& 3} and use them to prove Lemma \ref{reduction}. 
\emph{Phases 2 \& 3} take place in the more general setting of Lemma \ref{reduction}.

During \emph{Phase 1} we use out-arcs and in-arcs that have been revealed during the time intervals $(m_1,m_2]$ and $(m_2,m_3]$ respectively in order to show that w.h.p.\@ there exists a matching in $D_c$ consisting of at most $2\log n $ cycles spanned by $E_{m_3}$. By matching we refer to a complete matching i.e.\@ some $M \subseteq  V_c \times V_c \backslash \{(v,v):v \in V_c\} $ where every vertex has in- and out-degree exactly 1.

Thereafter, we randomly partition $E'=E^2\cup E^3$. In \emph{Phase 2}, we attempt to sequentially  join any two cycles  found in the current matching, starting with the matching above, to a single one. We join the cycles by a straightforward two-arc exchange,
where arcs $vw,xy$ in two distinct cycles are rerouted via $vy,xw$
if the latter two are in $E^2$ (illustrated at Figure 2). We show that once this is no longer possible, we are left with a large cycle consisting of $n-o(n)$ vertices of $D_\tau$.

Finally, during \emph{Phase 3} using arcs found in $E^3$, we sequentially try to merge the smaller cycles with the largest one. To merge two cycles here we start by finding an arc in $E^3$ joining them. This creates a dipath spanning the vertices of the two cycles. Afterwards, we grow the set of dipaths using ``double rotations", or sequences of two-arc exchanges that maintain a dipath on the same vertex set. (More specifically, for a dipath $P=(p_1,p_2,...,p_s)$, suppose $p_sp_k, p_{k-1}p_l \in E^3$ with $k < l$. Then a double rotation, illustrated at Figure 2, using those two arcs   replaces $P$ with the dipath  $P'=(p_1,p_2,...,p_{k-1},p_l,p_{l+1},...,p_s,p_k,p_{k+1},...,p_{l-1})$.)
By performing sequences of double rotations we find  $\Omega(n)$ paths with a common starting vertex but \emph{distinct} endpoints. With this many paths we succeed in closing one of them (joining the end-vertex to the start-vertex by an arc) with probability at least $1-o(n^{-\epsilon})$ for some $\epsilon >0$. Hence we may join all $(\leq 2 \log n)$ cycles inherited from \emph{Phase 2}.

\begin{figure}[!h]\label{pp}
\centering
\begin{tikzpicture}[yscale=1]

\draw[->] (2.78,1.2) arc (15:178:0.8);
\draw[->] (1.2,1) arc (180:345:0.8);
\draw[->][fill] (2.78,1.2) circle [radius=0.4mm];
\draw[->][fill] (2.78,0.85) circle [radius=0.4mm];
\draw[->,dotted,red](2.78,0.87)--(2.78,1.18);

\draw[->] (4.8,1) arc (0:165:0.8);
\draw[->] (3.22,0.85) arc (195:358:0.8);
\draw[->][fill] (3.22,0.86) circle [radius=0.4mm];
\draw[->][fill] (3.22,1.18) circle [radius=0.4mm];
\draw[->,dotted,red](3.22,1.18)--(3.22,0.86);

\draw[->,thick,blue](2.8,0.85)--(3.2,0.86);
\draw[->,thick,blue](3.2,1.18)--(2.8,1.18);

\draw node at (2.55,0.8) {$v$};
\draw node at (2.55,1.2) {$w$};
\draw node at (3.5,1.2) {$x$};
\draw node at (3.5,0.8) {$y$};

\draw[fill](6,0.8) circle [radius=0.4mm];
\draw[fill](6.4,0.8) circle [radius=0.4mm];
\draw[fill](6.8,0.8) circle [radius=0.4mm];
\draw[fill](7.2,0.8) circle [radius=0.4mm];
\draw[fill](7.6,0.8) circle [radius=0.4mm];
\draw[fill](8,0.8) circle [radius=0.4mm];
\draw[fill](8.4,0.8) circle [radius=0.4mm];
\draw[fill](8.8,0.8) circle [radius=0.4mm];
\draw[fill](9.2,0.8) circle [radius=0.4mm];
\draw[fill](9.6,0.8) circle [radius=0.4mm];
\draw[fill](10,0.8) circle [radius=0.4mm];
\draw[fill](10.4,0.8) circle [radius=0.4mm];
\draw[fill](10.8,0.8) circle [radius=0.4mm];

\draw[->,thick](6.0,0.8)--(6.38,0.8);
\draw[->,thick](6.4,0.8)--(6.78,0.8);
\draw[->,dotted,thick,red](6.8,0.8)--(7.18,0.8);
\draw[->,thick](7.2,0.8)--(7.58,0.8);
\draw[->,thick](7.6,0.8)--(7.98,0.8);
\draw[->,thick](8.0,0.8)--(8.38,0.8);
\draw[->,thick](8.4,0.8)--(8.78,0.8);
\draw[->,thick](8.8,0.8)--(9.18,0.8);
\draw[->,thick,dotted,red](9.2,0.8)--(9.58,0.8);
\draw[->,thick](9.6,0.8)--(9.98,0.8);
\draw[->,thick](10.0,0.8)--(10.38,0.8);
\draw[->,thick](10.4,0.8)--(10.78,0.8);

\draw [->, thick, blue] (6.8,0.8) to [out=300,in=240] (9.58,0.78);
\draw [->, thick, blue] (10.8,0.8) to [out=120,in=60] (7.22,0.82);

\node[below] at (7.2,0.8) {$p_{k}$};
\node[below] at (9.6,0.8) {$p_{l}$};
\node[below] at (10.8,0.8) {$p_s$};

\end{tikzpicture}
\caption{Left-Merging two cycles (\emph{Phase 2}),
Right-Double rotation (\emph{Phase 3}).}
\end{figure}
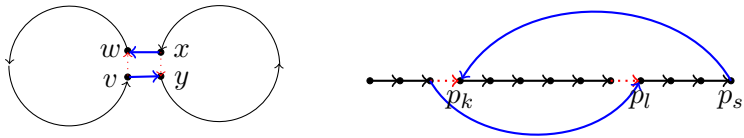

\section{Construction of $D_c$}\label{s.Dcalg}
Let $D_c'$ be the graph induced by the arcs of color $c$, $BAD=\{z_1,z_2,...,z_b\}$ where for some $s \leq b $ we have that $SMALL \cap BAD=\{z_1,z_2,...z_s\}$. $D_c$ is set to be the graph that we get after applying the following algorithm to $D_c'$.
We aim to thread all $BAD$ vertices, one at a time, into disjoint directed paths (we will later contract) with neither endpoint in $BAD$. We achieve this by dynamically keeping track of all potential starting vertices $V^+$ and potential ending vertices $V^-$ of these paths. It is likely that some $BAD$ vertices will have been used as endpoints of paths for other $BAD$ vertices before they had their turn-see the ``if'' clause below-but, in this case, we only need to extend the path in a single direction. 

\begin{algorithm}[H]
\caption{HideBad}
{$V^+$:$=V_n$, $V^-$:$=V_n$, $E_{contr}$:$=\emptyset.$\\}
\For{ $\ell=1,2,...,s$}{
Let $j,k \in [n]$ be minimal such that $v_j\in V^+, v_k \in V^-$ and $v_jz_\ell,z_\ell v_k \in E(D_c')$
\\$V^+ \leftarrow V^+ \backslash \{z_\ell,v_j\}, V^-\leftarrow V^- \backslash \{ z_\ell, v_k\}, E_{contr}\leftarrow E_{contr} \cup \{v_jz_\ell, z_\ell v_k\}.$
}
\For{ $\ell=s+1,s+2,...,b$}{
\uIf{ $z_\ell \notin V^+$ }
{ Let $j \in [n]$ be the minimum such that $v_j \in V^+$ and $v_jz_\ell  \in E(D_c')$ 
 \\$V^+ \leftarrow V^+ \backslash\{v_j\}, V^-\leftarrow V^- \backslash \{ z_\ell\}, E_{contr}\leftarrow E_{contr} \cup \{v_jz_\ell\}.$}
\uElseIf{ $z_\ell \notin V^-$ }
{ Let $k \in [n]$ be the minimum such that $v_k \in V^-$ and $z_\ell v_k \in E(D_c')$ 
   \\$V^+ \leftarrow V^+ \backslash \{z_\ell\}, V^-\leftarrow V^- \backslash \{v_k\}, E_{contr}\leftarrow E_{contr} \cup \{z_\ell v_k\}.$}
\Else{
Let $j,k \in [n]$ be minimal such that $v_j\in V^+, v_k \in V^-$ and $v_jz_\ell,z_\ell v_k \in E(D_c')$
\\$V^+ \leftarrow V^+ \backslash \{z_\ell,v_j\}, V^-\leftarrow V^- \backslash \{ z_\ell, v_k\}, E_{contr}\leftarrow E_{contr} \cup \{v_jz_l, z_\ell v_k\}.$
}}
Delete all arcs $xy$ in $E(D_c')\backslash E_{contr}$ such that $x \notin V^+$ or $y \notin V^-$.
\\Contract all edges in $E_{contr}$ and let $D_c$ be the resultant graph.
\end{algorithm}

It should not be obvious at this stage that we can always perform this algorithm so greedily, as one could feasibly run out of potential out- or in-neighbours of a given $z_\ell \in BAD$ at some late stage, all taken up by earlier $BAD$ vertices. We will devote the rest of this section to showing (Theorem \ref{hterm}) this is unlikely to be a problem (after reassuring ourselves that Hamiltonicity is preserved under these contractions in Lemma \ref{11ham}).

\begin{rem}
At each step of the algorithm $x\in V_n$ is removed from $V^+$ (similarly from $V^-$) iff for some $y \in V_n$ the arc $xy$ ($yx$ respectively) is added to $E_{contr}$.
\end{rem}

\begin{notn}
Henceforth we denote by $V_c$ the vertex set of $D_c$.
\end{notn}
\begin{defn}
For $v \in V_c$   set $contr(v)$:=$\{u \in V(D_c')$\@:$ \text{ $u$ gets contracted to $v$}\}$. Furthermore set $v^+$ and $v^-$ to be the unique elements found in  $contr(v) \cap V^+$ and  $contr(v) \cap V^-$ respectively.
\end{defn}
\begin{rem}\label{nbad}
Every $v \in V_c$ has both $v^+,v^- \notin BAD$. 
Furthermore $V^* := V\setminus(BAD\cup N(BAD))\subseteq V^+,V^-$.
\end{rem}

\begin{lem}\label{11cor}
For $u,v \in V_c$ we have that $uv \in E(D_c)  \Leftrightarrow u^+v^-\in E(D_c')$.
\end{lem}
\begin{proof}
Observe that $xy \in E(D_c')$ was removed or contracted iff after the last iteration of \emph{HideBad} $x\notin V^+$ or $y \notin V^-$. Let $u,v \in V_c$ be such that $u^+v^-\in E(D_c')$. Then since $u^+\in V^+$ and $u^-\in V^-$, from the observation follows that  $u^+v^-$ was not removed or contracted. In addition  $u^+v^-$ is identified with $uv$ after the contractions, hence $uv \in E(D_c)$.  Let $a,b \in V_c$ be such that  $ab \in E(D_c)$ so certainly $a \neq b$. $ab$  originated from an edge in $\big(contr(a)\times contr(b)\big) \cap E(D_c')$ and since any edge in  $\big(contr(a)\times contr(b)\big)\backslash \{a^+b^-\}$ was either contracted or removed  it must be the case that $ u^+v^-\in E(D_c')$.
\end{proof}
\begin{lem}\label{11ham}
If there exists a Hamilton cycle in $D_c$ then there exists a Hamilton cycle in $D_c'$.
\end{lem}
\begin{proof}
For $u \in V_c$ define  $P(u)$  to be the dipath in $D_c'$ that contains all the vertices in $contr(u)$, starts at $u^-$, ends with $u^+$ and uses all the arcs in $E_{contr}$ that are spanned by $contr(u)$ (in the case that $\vert contr(u)\vert =1$, $P(u)$ is a single vertex i.e.\@ a dipath of length 0).  Now suppose $v_{\pi(1)},v_{\pi(1)}v_{\pi(2)},v_{\pi(2)},$ $...,v_{\pi(n_c)},v_{\pi(n_c)}v_{\pi(1)},v_{\pi(1)}$ is a  Hamilton cycle in $D_c$ then, we have that $P(v_{\pi(1)}),v_{\pi(1)}^+v_{\pi(2)}^- ,$ $P(v_{\pi(2)}),...,$ $P(v_{\pi(n_c)}),$   $v_{\pi(n_c)}^+v_{\pi(1)}^-,P(v_{\pi(1)}^-)$ is a Hamilton cycle in $D_c'$. To see this, first note that $P(v_{\pi(i)})$ starts with $v_{\pi(i)}^-$ and ends with
$v_{\pi(i)}^+$. Moreover $v_{\pi(i)}v_{\pi(i+1)}\in E(D_c)$ implies, by Lemma \ref{11cor}, that $v_{\pi(i)}^+v_{\pi(i+1)}^-\in E(D_c')$. Finally, since the sets $contr(v)$ partition $V_n$, each vertex in $V_n$ appears exactly in one of the dipaths $P(u)$.
\end{proof}
\begin{tm}\label{hterm}
W.h.p. the algorithm $HideBad$ terminates.
\end{tm}
The proof of Theorem \ref{hterm} will follow from Lemmas \ref{Badapart} and \ref{neib} proven in this section. To state and prove these we will need the following definitions.
\begin{defn}
For $v \in V_n$, let $N(v)$:=$\{u \in V_n: d_\tau'(u,v)=1\}$ (i.e those vertices whose undirected distance from $v$ is one). Similarly set $N(N(v))$:=$\{u \in V_n: d_\tau'(u,v) \in \{1,2\} \}$.
\end{defn}
\begin{rem}\label{b}
All  three sets of edges that appear at times found in $(0,m_1], (m_1,m_2]$ and $(m_2,m_3]$ respectively are distributed as the edges of $D_{n,m_1}$. Hence, by additionally taking into account the symmetry between in- and out- arcs in $D_{n,m_1}$,  the sets $B_1^+, B^-_1, B_2^+$ and $B^-_3$ (defined during the execution of \emph{COL}) follow the same distribution.
\end{rem}
\begin{lem}\label{Badapart}
 W.h.p.\@ for all $v \in V_n$  we have that $|BAD \cap N(N(v))| \leq 4e^{q\cdot 10^5}$.
\end{lem}
\begin{proof}
Let $k=e^{q\cdot10^5}$ and suppose $\vert BAD\cap N(N(v)) \vert>4k$ for some $v\in V_n$. Then there is some digraph $S\subseteq D_\tau$ with $V(S)=\{v,b_1,...,b_k,w_1,...,w_l\}$ for some $l\leq k$ satisfying the following . For some  $i\leq k$ all of the vertices
$b_1,...,b_{i},w_1,...,w_l$ are connected to $v$ by  arcs $e_1,...,e_{i+l}$ and for $i< j \leq k$,  $b_j$ is connected to some $v_j \in \{b_1,...,b_{i},w_1,...,w_l\}$ by the arc $e_{j+l}$. Furthermore there is some $B^*\in \{B_1^+, B_1^-, B_2^+, B_3^-\}$ such that $B=\{b_1,...,b_k\}\subseteq B^*$. Suppose $B^*=B^+_1$. By setting for $E \subseteq  E(S)$ the events
$S_{m_1}(E)$:=$\{E(S)\cap E_{m_1}=E\}$ and $S_{m_1,\tau}(E)$:=$\{E(S)\backslash E \subseteq E_\tau\backslash E_{m_1}\}$ we have,
\begin{align}\label{mmm}
\begin{split}
L=\mathbb{P}&\big(\{S \subseteq D_\tau\} \wedge \{B \subseteq B_1^+\}\big)
=\sum_{E \subseteq  E(S)}\mathbb{P}\big(S_{m_1}(E) \wedge S_{{m_1},\tau}(E) \wedge \{B\subseteq B_1^+\}\big) \\
&=\sum_{E \subseteq  E(S)}\mathbb{P}\big(S_{m_1}(E)\big)\cdot \mathbb{P}\big(S_{{m_1},\tau}(E) \big\vert S_{m_1}(E) \big) \cdot \mathbb{P}\big(B\subseteq B_1^+ \big\vert S_{m_1}(E) \wedge (S_{{m_1},\tau}(E) \big).
\end{split}
\end{align}
For fixed $E \subseteq  E(S)$ (1) implies that,
\begin{align*}
    \mathbb{P}\big(S_{m_1}(E)\big) \leq 10\sqrt{m_1} p_1^{|E|}(1-p_1)^{\vert E(S)\backslash E \vert } \leq np_1^{|E|} 
    \leq n \bigg(\frac{\log n}{n} \bigg)^{|E|}.
\end{align*}
Furthemore,
\begin{align*}
     \mathbb{P}\big(S_{{m_1},\tau}(E) \big\vert S_{m_1}(E) \big)&= \frac{
     \binom{n(n-1)-m_1-\vert E(S)\backslash E \vert}{\tau-m_1-\vert E(S)\backslash E \vert}}{\binom{n(n-1) - m_1} {\tau-m_1}}
     =\frac{\binom{\tau-m_1}{|E(S)\backslash E|}}{\binom{n(n-1)-m_1}{|E(S)\backslash E|}} 
     =\prod_{i=0}^{|E(S)\backslash E|-1}\frac{\tau-m_1-i}{n(n-1)-m_1-i}\\
     &\leq \bigg(\frac{\tau-m_1}{n(n-1)-m_1}\bigg)^{|E(S)\backslash E|}
     \leq \bigg( \frac{2n\log n}{n^2}\bigg)^{|E(S)\backslash E|}.
\end{align*}
Finally, in order to bound  $\mathbb{P}\big(B\subseteq  B_1^+ \big\vert S_{m_1}(E) \wedge (S_{{m_1},\tau}(E) \big)$ from above note the following. There are $\binom{n(n-1)-|E(S)|}{m_1-|E|}$ ways to pick $E_{m_1}\backslash E$ so that it can be extended to a chain $E_{m_1}\backslash E \subseteq  E_{m_1}\backslash E_\tau$ such that $E_{m_1}$ and $E_\tau$ satisfy both the events  $S_{m_1}(E)$ and $S_{{m_1},\tau}(E)$. Given $S_{m_1}(E)$ and $S_{{m_1},\tau}(E)$ occur $E_{m_1}\backslash E$ is equally likely to be any of those   
$\binom{n(n-1)-|E(S)|}{m_1-|E|}$ choices.  
Moreover, if $B \subseteq  B_1^+$ then every vertex in $B$ has at most $\epsilon\log n$ out-arcs in $E_{m_1}$. Hence there are at most $f=\epsilon |B| \log n=\epsilon k \log n$ arcs in $E_{m_1}\backslash E$ with out-vertex in $B$ (i.e.\@ from the set $\{bv: b\in B, v\in V_n \text{ and } v\neq b\}$). Thus,

\begin{align*}
     \mathbb{P}&\big(B\subseteq  B_1^+ \big\vert S_{m_1}(E) \wedge S_{{m_1},\tau}(E) \big)
     \leq \frac{\underset{j=0}{\overset{f}{\sum}}
\binom{k(n-1)}{j}\binom{n(n-1)-k(n-1)}{m_1-\vert E\vert-j}}{\binom{n(n-1)-|E(S)|}{m_1-|E|}}
\leq \frac{f
\binom{k(n-1)}{f}\binom{n(n-1)-k(n-1)}{m_1-\vert E\vert-f}}{\binom{n(n-1)-|E(S)|}{m_1-|E|}}
\\
&\leq f \binom{k(n-1)}{f} \frac{(m_1-|E|)!}{(m_1-|E|-f)!}
\frac{\frac{[n(n-1)-k(n-1)]!}{[n(n-1)-k(n-1)-m_1+|E|+f]!}}{\frac{\big(\underset{j=0}{\overset{ f-1 }{\prod}}
{n(n-1)-|E(S)|-m_1+|E|+f-j}\big)
[n(n-1)-|E(S)|]!
}{[n(n-1)-|E(S)|-m_1+|E|+f]!}}
\\
&    \leq f
\bigg(\frac{ekn}{f}\bigg)^f
\prod_{j=0}^{f-1}\frac{m_1-|E|-j}{n(n-1)-|E(S)|-m_1+|E|+f-j}
\prod_{j=0}^{m_1-|E|-f-1} \frac{ n(n-1)-k(n-1)-j}{n(n-1)-|E(S)|-j}\\
&      \leq  f
\bigg(\frac{ekn}{f}\bigg)^f
\bigg(\frac{m_1}{0.9n^2}\bigg)^f
\cdot\prod_{j=0}^{m_1-|E|-f-1} \frac{ n(n-1)-k(n-1)}{n(n-1)-|E(S)|}\\
&      \leq  f
\bigg(\frac{ekm_1}{0.9fn}\bigg)^f
\exp\bigg\{-\frac{k(n-1)-|E(S)|}{n(n-1)} \cdot (m_1-|E(S)|-f-1)  \bigg\} \\
&\leq \epsilon k \log n 
\bigg(\frac{ekm_1}{0.9\epsilon k n \log n}\bigg)^{\epsilon k \log n}
\exp\bigg\{-\frac{0.8km_1}{n}        \bigg\} \leq \epsilon k \log n
\bigg(\frac{1}{\epsilon  }\bigg)^{\frac{m_1}{n}}
\exp\bigg\{-\frac{0.8km_1}{n}        \bigg\}
\\
& \leq \epsilon k \log n\cdot \exp\bigg\{[-\log(\epsilon )-0.8k]\frac{m_1}{n}  \bigg\} 
\leq \exp\bigg\{ -0.7 k\cdot \frac{m_1}{n} \bigg\} 
\\&
\leq \exp\bigg\{ -0.7 e^{q\cdot 10^5} e^{-q\cdot10^4}\log n \bigg\}
\leq \exp\bigg\{ - e^{8.9q\cdot 10^4}\log n \bigg\}
.
\end{align*}

The 2nd inequality follows from the fact that $\
\binom{k(n-1)}{j}\binom{n(n-1)-k(n-1)}{m_1-\vert E\vert-j}$
is increasing for $j \in [1, f]$.
Thus, using the upper bounds found for the quantities on the right hand side of (\ref{mmm}) we obtain
\begin{align*}
L & \leq\sum_{E \subseteq  E(S)} n \bigg(\frac{\log n}{n}\bigg)^{|E|}
\cdot\bigg( \frac{2\log n}{n}\bigg)^{|E(S)\backslash E|}   \cdot \exp\bigg\{ - e^{8.9q\cdot 10^4}\log n \bigg\}  \\
& \leq \sum_{E \subseteq  E(S)}  \bigg(\frac{\log n}{n}\bigg)^{|E(S)|}\exp\bigg\{-e^{8.8q\cdot 10^4} \log n\bigg\} 
\leq \bigg(\frac{\log n}{n}\bigg)^{|E(S)|}\exp\bigg\{{-e^{8q\cdot 10^4} \log n}\bigg\}.
\end{align*}
For fixed $l, k $ there are exactly $n\binom{n-1}{k}\binom{n-1-k}{l}$ ways to choose the vertices of $S$, or equivalently, disjoint sets  $\{v\}, \{b_1,...,b_k\}$ and $\{w_1,...,w_l\}$ from $V_n$. Thereafter there are at most  $2^{l+k}\sum_{i=0}^{k}\binom{k}{i} (i+l)^{k-i}$  choices for its directed edges.
Taking into account Remark \ref{b} and that $l\leq k = e^{q\cdot 10^4}$, union bound gives us
\begin{align*}
\mathbb{P}\big(\exists v& \in V_n:\vert BAD \cap N(N(v)) |> 4e^{q\cdot 10^5}  \big)\\
&\leq \mathbb{P}\big( \exists v\in V_n \text{ and } (i,*) \in \{(1,+),(1,-),(2,+),(3,-)\}:B^*_i \cap N(N(v)) > e^{q\cdot 10^5} \big)\\
&\leq 4\sum_{l=0}^{k}n\binom{n-1}{k}\binom{n-1-k}{l}2^{k+l}\sum_{i=0}^{k}\binom{k}{i} (i+l)^{k-i}\bigg(\frac{\log n}{n}\bigg)^{l+k}\exp\bigg\{{-e^{8q\cdot 10^4} \log n}\bigg\}\\
&\leq 4\sum_{l=0}^{k} n^{l+k+1}\bigg(\frac{\log n}{n}\bigg)^{l+k}\exp\bigg\{{-e^{7q\cdot 10^4} \log n}\bigg\}=o(n^{-2}).\qedhere
\end{align*}
\end{proof}

\begin{lem}\label{full}
W.h.p.\@ for every  $u \notin BAD$ we have that $u \in FULL_{m_1}^+ \cap FULL_{m_1}^-$.
\end{lem}
\begin{proof}
With $k=4e^{q\cdot10^5}$
Lemma \ref{Badapart} implies that w.h.p.\@ for every $u\in V_n$ we have $\vert \{ w \in V_n: w\in N^+(u),$ $h(uw)<m_1$ and $d^+_{m_1}(w) \leq \epsilon \log n\}\vert  \leq k$. Hence as $u \notin BAD$ implies that $d=\min \{d_{m_1}^+(u),d_{m_1}^-(u)\}  \geq \epsilon\log n$ from
 Remark \ref{ss}, with $\gamma=\epsilon$ it follows that
\begin{align*}
\mathbb{P}\big(\exists u \notin BAD \text{ s.t.\@ } u \notin FULL_{m_1}^+ \cap FULL_{m_1}^-\big)
&\leq 2n \underset{\epsilon \log n \leq d \leq n}{\max} \bigg\{\binom{d-k}{q-1}\bigg(\frac{q^{q+1}}{\epsilon \log n}\bigg)^{(d-k)-(q-1)}\bigg\}\\
&\leq 2n n^{q-1} \bigg(\frac{q^{q+1}}{\epsilon\log n}\bigg)^{0.5\epsilon \log n}=o(1).\qedhere
\end{align*}
\end{proof}
\begin{lem}\label{neib}
 W.h.p.\@ for every $v \in BAD \backslash SMALL$ we have that $v$ has at least $\log \log n$ out-arcs in each color ending in $V_n\backslash BAD$ and at least $\log \log n$ in-arcs in each color starting from $V_n\backslash BAD$.
\end{lem}
\begin{proof}
Let $v \in BAD \backslash SMALL$. Then $v$ has at least $\frac{\log n}{100}$ out-neighbors. Lemma \ref{lem4} gives us that the out- degree of $v$ at time $m_3$ is at most $\frac{3\log n}{10^3q}$. Therefore  $v$ has at least  $\frac{\log n}{100}  -  \frac{3\log n}{10^3q} -  4e^{q\cdot 10^5}$ out-neighbors in $V_n \backslash BAD$ that arrive after $m_3$. By the previous Lemma w.h.p.\@  for all $u \in V_n \backslash BAD$ and all $c \in [q]$ we have $d^-_{m_1}(u,c)\geq 1$. Hence at most $q$ such arcs  $vu$ that arrive at some time $t> m_3$ will be colored under the condition $ v \notin FULL_{t-1}^+ $. Thus there are at least $\frac{\log n}{100}  -  \frac{3\log n}{10^3q} -  4e^{q\cdot 10^5}-q$ arcs $vu $ with $u \in V_n \backslash BAD$ that will arrive at some time $t> m_3$ and will be colored with color $c$ that minimizes $d^+_t(v,c)\mathbb{I}\{v \in BAD\}
+d^-_t(u,c)\mathbb{I}\{u \in BAD\}= d^+_t(v,c)$ (i.e.\@ the arcs are given a color in which $v$ has the smallest out-degree when they appear). Thus $v$ will have  at least $\frac{1}{q}\big (\frac{\log n}{100}  -  \frac{3\log n}{10^3q} -  4e^{q\cdot 10^5}-q\big) -1 \geq \log \log n$ out-arcs in each color ending in $V_n\backslash BAD$. A similar argument holds for the number of arcs from $V_n \backslash BAD$ to $v$.
\end{proof}
{\bf{Proof of Theorem \ref{hterm}}}. Assume that the algorithm $HideBad$ does not terminate. Then there is an iteration $f$ at which there  do not exist $v_j \in V^+$ and $v_k \in V^-$ such that $v_jz_{f},z_{f}v_k\in E(D_c')$,
WLOG the former (the case $\nexists v_k\in V^-$ will follow similarly).
\vspace{5mm}
\\Case 1: $f\leq s$ (i.e $z_{f}\in SMALL$).  As every vertex has in-degree at least one in $D'_c$, there exists
$x\in V_n$ such that the arc $xz_f$ belongs to $ E_{\tau}$ and has color $c$. Hence, $\exists \ell< f$ such that  at $\ell$-th iteration $x$ was removed from $V^+$.
This implies that $z_{\ell}\in N(N(z_f))$. Hence we get that $z_\ell, z_f$ belong to $SMALL$ and $z_f,z_\ell$ have distance less than 3 contradicting Lemma \ref{lem2}.
\vspace{5mm}
\\Case 2: $s < f \leq b$ (i.e $z_{f}\in BAD \backslash SMALL$). Since $z_f \notin SMALL$ Lemma \ref{neib}
implies that  $\exists S \subseteq  V_n$  such
that $\vert S \vert \geq \log \log n$ and for  every $z\in S$ the arc $zz_f$ belongs to $ E_\tau$ and has color $c$.
Observe that at any iteration $\ell < f$  at most 2 vertices are removed from $V^+\cap S$ in the case that $z_\ell \in N(N(z_f))$, and none are removed otherwise. Hence as $V^+\cap S=\emptyset$ at the beginning of the $f$-th iteration we have that  $2|N(N(z_f)) \cap BAD| \geq \log \log n$ which contradicts Lemmas \ref{Badapart} and \ref{neib}.
\section{Structure of $D_c$}
\begin{lem}\label{sizebad}
 W.h.p.\@ $|BAD| = o(n^{1-\delta})$,  for some constant $\delta >0$.
\end{lem}
\begin{proof}
Recall that $p_1=m_1/n(n-1)$. For every $v \in V_n$, (2) gives us
\begin{align*}
\mathbb{P}\big(v \in BAD\big)& = \mathbb{P}\big(v \in B_1^+ \cup B_1^-\cup B_2^+ \cup B_3^-\big) \leq 4\mathbb{P}(v\in B^+_1)= 4 \mathbb{P}\bigg(d^+_{m_1}(v)\leq  e^{-\epsilon \log n} \bigg)\\
& \leq 4\cdot 3\cdot \mathbb{P} \bigg( Bin(n-1,p_1)\leq  \epsilon  \log n \bigg) 
\leq 12\exp \bigg( -0.49e^{-q \cdot 10^4}  \log n\bigg) = n^{-0.49e^{-q \cdot 10^4 }}. 
\end{align*}
At the last inequality we used (3).
Hence by Markov's inequality, we have
\begin{align*}
\mathbb{P}\bigg(|BAD| > n^{1-0.4e^{-q \cdot 10^{4} }}\bigg) &\leq \frac{\mathbb{E}(|BAD|)}{n^{1-0.4e^{-q \cdot 10^4 }}}
\leq n^{-0.09e^{-q \cdot 10^{4} }}. \qedhere
\end{align*}
\end{proof}

\begin{lem}\label{sizevert}
W.h.p.\@ $|V_c| = n-o(n)$.
\end{lem}
\begin{proof}
Every contraction that occurs during  the execution of $HideBad$ reduces the number of vertices by one. As  at most $2|BAD|$ contractions are performed, Lemma \ref{sizebad} gives us that w.h.p.\@ $|V_c| \geq n- 2\cdot n^{1-0.4e^{-q \cdot 10^{4} }}$.
\end{proof}
We henceforth set $n_c:=\vert V_c \vert =(1-o(1))n$.
\section{PHASE 1}

In this section we take our first step toward proving that w.h.p.\@ $D_c$ has a Hamilton cycle  by showing that w.h.p.\@ there exists a matching in $D_c$ consisting of at most $2 \log n$ cycles and whose edges appear by time $m_3$. As usual, we proceed by implicitly conditioning on all aforementioned events proven to occur w.h.p.
In the proof of Lemma \ref{6edges} we are going to use the following elementary result.
\begin{lem}\label{singlecycle}
W.h.p.\@ in $D_{\tau_q}$ no vertex belongs to two distinct cycles of length at most 4.
\end{lem}
\begin{proof}
In the event that there is a vertex that belongs to two distinct cycles of length at most 4
 there are $3\leq k\leq 7$ vertices that span $k+1$ edges in $D_{\tau_q}$. Since w.h.p.\@ $\tau_q<2\log n$, (2) implies that the  probability of such event occurring is bounded by
\begin{align*}
3\sum_{k=3}^7\binom{n}{k} \binom{k(k-1)}{k+1} \bigg(\frac{2\log n}{n} \bigg)^{k+1}=o(1).
\end{align*}
\end{proof}
\begin{lem}\label{6edges}
 W.h.p.\@  every  $v \in V_c$ has at least 6 out- and 6 in- arcs in $E(D_c)$  revealed during the intervals $(m_1,m_2]$ and $(m_2,m_3]$ respectively, whose other endpoint lies in $V^*:=V \backslash (BAD \cup N(BAD))$.
\end{lem}
Here, it is imperative that we avoid $BAD \cup N(BAD)$, since those vertices have already been assigned an edge in at least one direction by the algorithm $HideBad$ from Section \ref{s.Dcalg}.

\begin{proof}
We originally defined $BAD$ during the algorithm $COL$ to make sure these vertices we want to work with had many edges during the $(m_1,m_2]$ period, and the cycling between colors means a positive proportion of them obtain color $c$.
The edges to $BAD$ don't enjoy the cyclic colors, and the edges to $N(BAD)$ are discarded altogether even if they were in desired color $c$, but the estimates from Section \ref{s.Dcalg} forbid too many of these vertices from being clustered around $v$.
\\
More explicitly,  let $v\in V_c$. Then by Remark (\ref{nbad}) we have $v^+ \notin BAD$, therefore
 Lemma \ref{full} gives us  $v^+ \in FULL_{m_1}^+$.    Now $v^+ \not\in BAD \Rightarrow v^+ \not\in B_2^+$,  so there are at least $\epsilon \log n$ arcs $v^+w$, $w \in V_n$ that have been revealed after the time $m_1$ and before the time $m_2+1$. 
Any such arc $v^+w$ that was \emph{not} colored cyclically was due to $w \not\in  FULL_{m_1}^+$ taking priority, and hence $w \in N(v^+) \cap BAD$ by Lemma \ref{full}.  
So out of all the potential arcs at least $\frac{1}{q}(\epsilon \log n -|BAD \cap N(v^+)|)-1$ have color $c$ (see lines 6-14, 24-25 of \emph{COL}), 
 and already none of these are to $BAD$.
Meanwhile, for $N(BAD)$, Lemma \ref{Badapart} immediately gives $|N(N(v^+))\cap BAD|\leq 4e^{q\cdot 10^5}$.
In addition Lemma \ref{singlecycle} implies that $\forall w\in BAD$, $|N(v^+)\cap N(w)|\leq 2$,
so any $w \in N(N(v^+))$ arose from $\leq 2$ neighbours of $v^+$, and  it follows $|N(BAD) \cap N(v^+)|\leq 2 \cdot 4e^{q\cdot 10^5}$.
Hence, since $ V^*:= V\setminus (BAD \cup N(BAD))$, 
there are at least $\frac{\epsilon \log n}{q}-\frac{8}{q}e^{q\cdot 10^5}-1 -4e^{q\cdot 10^5}\geq 6$ arcs from $v^+$ to $V^*$ in $E(D_c)$ revealed during the interval $(m_1,m_2]$.

The other part of this Lemma follows in a similar fashion (with $v^-$, $FULL^-_{m_1}$ and $(m_2,m_3]$ in place of  $v^+$,  $FULL^+_{m_1}$ and $(m_1,m_2]$ respectively).
\end{proof}

\begin{defn}
For $v \in V_c$ set:
$$E_c^+(v):=\{ \text{the first six arcs from } v \text{ to $V^*$ in } E(D_c) \text{ that are revealed in } (m_1,m_2]\},$$
 $$E_c^-(v):=\{ \text{the first six arcs from } v \text{ to $V^*$ in } E(D_c) \text{ that are revealed in } (m_2,m_3]\},$$
$$E_c^+:= \underset{v \in V_c}{\cup }E_c^+(v), \hspace{20mm} E_c^-:= \underset{v \in V_c}{\cup }E_c^-(v).$$
\end{defn}
From Lemma \ref{6edges} it follows that w.h.p.\@ the above sets are well-defined.
\begin{lem}\label{matching}
 W.h.p.\@ $E_c^+ \cup E_c^-$ spans a matching on $V_c$ consisting of at most $2\log n_c$ cycles.
\end{lem}
\begin{proof}
We will first show that  w.h.p.\@ $E_c^+ \cup E_c^-$ spans a matching on $V_c$. Assume that $E_c^+,E_c^-$ do not span a matching. Then HALL's Theorem gives us that there exists $K \subseteq  V_c$ with $|K|=k \leq \frac{n_c}{2}$ that has in- or out-neighbourhood induced by $E_c^-$ and $E_c^+$ respectively of size $k-1$. We will examine the case of its out-neighborhood being of size $k-1$. The other case will follow in a similar fashion.
\vspace{3mm}
\\Let $Y^+$  be the random subgraph of $D_c$ with edge set $E(Y^+):=E_{m_3}\backslash E^+_c$. Conditioned on $E(Y^+)$ we may assume that  for every $v \in V(D_c)$, $E_c^+(v)$ has been chosen independently uniformly at random from all sets of arcs form $v$ to  $V^*\backslash \{v\}$  of size 6 that have empty intersection with $E(Y^+)$.
To see this let $E(Y^+)=\{f_1,...,f_k\}$, $h_1,...,h_k \in[m_3]$ and for $v \in V_c$ we let $H_v\subseteq [m_3]$ such that $|H_v|=6$. If we further conditioned on the event $\mathcal{E}=\bigg(\underset{i\in[k]}{\bigwedge}\{h(f_i)=h_i\}\bigg) \wedge \bigg(\underset{v\in V_c}{\bigwedge}\big\{ \{h(e):e \in E_c^+(v)\}=H_v\big\}\bigg)$, in the case $\mathcal{E} \neq \emptyset$,
we have that for any $w \in V_c$ each set of arcs from $w$ to  $V^*\backslash \{w\}$   of size 6 that has empty intersection with $E(Y^+)$ has the same probability to be $E_c^+(w)$. Moreover the identity of the edges in $E_c^+(w)$ does not depend on the identity of $\{E_c^+(u):u \in A\}$ for any $A \subseteq  V_c \backslash \{w\}$.
\\We write $d^+_{Y^+}(v,S)$ for the number of arcs in $Y^+$ from $v$ to a given $S \subseteq  V_c$. Lemma \ref{lem4} implies that for every $v \in V_c$, $d^+_{Y^+}(v,V_c)\leq \frac{3\log n}{10^3q}$.  Therefore the probability of having a set $K\subseteq  V_c$ that has as out neighborhood induced by $E^+_c$ a set $S\subseteq  V^*$ with $6\leq |S|=|K|-1 \leq \frac{n_c}{2}$ is bounded above by
\begin{align*}
\sum_{k=7}^{\frac{n_c}{2}} &\sum_{|K|=k}\sum_{|S|=k-1}\prod_{v \in K}
\binom{k-1-\mathbb{I}(v \in S)-d^+_{Y^+}(v,S)}{6}\bigg/ \binom{|V^*|-1- d^+_{Y^+}(v,V)   }{6} \\
& \leq \sum_{k=7}^{\frac{n_c}{2}} \binom{|V_c|}{k}\binom{|V^*|}{k-1}   \prod_{j=1}^{k}\binom{k}{6}\bigg/ \binom{|V^*|-1- \frac{3\log n}{100q}}{ 6}
  \leq  \sum_{k=7} ^{\frac{n_c}{2}} \bigg(\frac{3n_c}{k}\bigg)^{2k}   \prod_{j=1}^{k}\frac{k^6}{(1-o(1))n_c^6}\\
 &\leq \sum_{k=7} ^{\frac{n_c}{2}} \bigg(\frac{3^2k^6n_c^2}{(1-o(1))k^2n_c^6}\bigg)^k
\leq  \sum_{k=7}^{\frac{n_c}{2}}\bigg(\frac{8k^4 }{(1-o(1))n_c^4}\bigg)^k =o(1).
\end{align*}
At the second inequality we used that Lemmas \ref{maxdegree} \& \ref{sizebad} imply that $n_c=|V_c|\geq |V^*|\geq |V|-|BAD|-|N(BAD)|=(1-o(1))n=(1-o(1))n_c.$
Hence, Hall's condition fails with probability $o(1)$ and w.h.p.\@ $E_c^+\cup E_c^-$ spans a matching.
\vspace{3mm}
\\We proceed to prove that a random matching spanned by $E_c^+ \cup E_c^-$ consists of at most $2\log n_c$ cycles.
 First let $W$ be the number of cycles that span less than 2 vertices of $V^*$
(i.e. 2-cycles  of the form $v,w$ with $v\in V^*$ and $w \notin V^*$). Then 
\begin{align*}
\mathbb{P}(W\geq 1)
&\leq \sum_{v\in V^*}\sum_{w \notin V^*} \mathbb{P}(vw^+ \in E^+_c(w) \text{ and } w^-v \in E^-_c(w)) 
\\ &\leq |V^*| |BAD|  \bigg(\frac{6}{(1+o(1))|V^*|} \bigg)^2 =o(1).
\end{align*}
Let $M$ be a random matching spanned by $E_c^+ \cup E_c^-$. 
Since w.h.p.\@ there is no such cycle spanned by a single vertex of $V^*$
we have that w.h.p.\@ $M$ induces a derangement on $V^*$. 
Finally conditioned on $V^*$, due to the symmetry of the edges with an endpoint in $V^*$, any such derangement is equally likely to occur.

Indeed let $A\subseteq V$ and consider any valid edge sequence ${\cal{E}}=e_1,...,e_{\tau_q}$. 
Let $\phi_1,\phi_2$ be any two permutations on $V$ that act as the identity on $V\setminus A$. Also let $\rho= \phi_2 \phi_1^{-1}$. Finally set ${\cal{E'}}=e_1',...,e_{\tau_q}'$ where for $i\in [\tau_q]$ $e_i=(u_i,w_i)$ and $e_i'=(u_i,\rho(w_i))$. 
Note that, provided $V\setminus A$ contains all $SMALL$ vertices, $ \cal{E'}$ is also a valid edge sequence.
Denote by $BAD_{\cal{E}}, V_{D,{\cal{E}}}, V_{\cal{E}}^+, V_{\cal{E}}^-, E_{c,{\cal{E}}}^+, E_{c,{\cal{E}}}^-$ and  $BAD_{\cal{E}'},V_{D,{\cal{E}}}, V_{\cal{E}}^+, V_{\cal{E}'}^-, E_{c,{\cal{E}'}}^+, E_{c,{\cal{E}'}}^-$ the sets $BAD, V_D, V^+, V^-, 
 E_{c}^+, E_{c}^-$ as defined by the sequences $\cal{E}$ and $\cal{E}'$ respectively. 
 
First assume that $A=V^*_{\cal{E}}$.
Then, as $\rho$ acts on the in-vertices of arcs with
in-vertex in $A$, we have  $A=V^*_{\cal{E}'}$. Similarly, by considering $\rho^{-1}$ we have  $A=V^*_{\cal{E}}$ only if  $A=V^*_{\cal{E}'}$.
Hence $A=V^*_{\cal{E}}$ iff $A=V^*_{\cal{E}'}$. Thereafter, given that 
$A=V^*_{\cal{E}} = V^*_{\cal{E}'}$,
 we have $BAD_{\cal{E}}= BAD_{\cal{E}'}$ and by extension, since the arcs adjacent to $BAD$ vertices are the same and appear in the same order in both sequences, we have
$V_{D,{\cal{E}}}=V_{D,{\cal{E}'}}$.
Furthermore $(u,w)\in  E_{c,{\cal{E}}}^+(u)$ iff $(u,\rho(w))\in  E_{c,{\cal{E}'}}^+(u)$ and $(u,w)\in  E_{c,{\cal{E}}}^-(w)$ iff $(u,\rho(w))\in  E_{c,{\cal{E}'}}^-(\rho(w))$. Therefore
$(u,w)\in  E_{c,{\cal{E}}}^+ \cup E_{c,{\cal{E}}}^-$ iff
$(u,\rho(w))\in  E_{c,{\cal{E}'}}^+ \cup E_{c,{\cal{E}'}}^-$.
Finally, given that 
$A=V^*_{\cal{E}} = V^*_{\cal{E}'}$,
 is not hard to check that $E_{c,{\cal{E}}}^+ \cup E_{c,{\cal{E}}}^-$ spans  a matching on $V_{D,{\cal{E}}}$ that induces the permutation $\phi_1$ on $A$ 
iff $E_{c,{\cal{E}'}}^+ \cup E_{c,{\cal{E}'}}^-$ spans a matching on $V_{D,{\cal{E}'}}$ that induces the permutation $ \rho(\phi_1)=\phi_2$ on $A$.
Here by induces we mean the following: if $u,u_k\in A$ and $u_1,u_2,...,u_{k-1} \notin A$ then the matching with arcs $(u,u_1),(u_1,u_2),...,(u_{k-1},u_k)$
induces a permutation on $A$ that sends $u$ to $u_k$. 

It is known (see for example \cite{cycles1}. \cite{cycles}) that the number of cycles, in a uniform random derangement on $[|V^*|]$, consists w.h.p.\@ of at most $2\log |V^*| \leq 2\log n_c$ cycles. Hence w.h.p.\@ $E^+_c \cup E^-_c$  spans a matching consisting of at most $2\log n_c$ cycles.
\end{proof}

\section{Reduction to Lemma \ref{reduction}}
Our vertex set is $V_c$. Lemma \ref{11ham} states that if $D_c$ is Hamiltonian then $D_{\tau_q}$ spans a cycle of color $c$. Hence,
in order to give a reduction of Theorem \ref{maintm}
to Lemma \ref{reduction} we need to define digraphs $F,H,D_{n_c,p}$ on $V_c$ such that:
\begin{enumerate}[i)]
\item $F$ is a 1-factor consisting of $O(\log n_c)$ directed cycles, 
\vspace{-2mm}
\item $H$ has total maximum in-/out- degree $O(\log n_c)$, 
\vspace{-2mm}
\item $D_{n_c,p}$ is a random digraph, every arc appears independently with probability $p=\Omega(\frac{\log n_c}{n_c})$ 
\vspace{-2mm}
\item  w.h.p.\@ $E(F),E(D_{n_c,p}) \subseteq D_{\tau_q}$ and all the arcs in $E(F) \cup (E(D_{n_c,p})\setminus E(H))$ have color $c$.
\end{enumerate}

 We let $F$ be a 1-factor spanned by $E^+_c\cup E^-_c$ consisting of at most $2\log n_c$ cycles, as provided by Lemma \ref{matching}.
We also let $H$ consist of all edges that appear by time $m_3$. Lemma \ref{lem4} implies that the maximum in/out-degree of $H$ is $O(\log n_c)$.

For the construction of $D_{n_c,p}$ we consider the arcs appearing in 
$(m_3,\tau_q]$. Since
\begin{itemize}
\item w.h.p.\@ $\tau_q- m_3\geq \frac{3}{4}\log n_c$,
\item w.h.p.\@ $|BAD| =o(n_c)$,
\item Every arc that appears after time $m_3$ and is not adjacent to $BAD$ is colored $c$ independently with probability $\frac{1}{q}$, and \item Every arc
in $D_c$ that has not appeared by time $m_3$ corresponds to exactly one arc not in $D_{m_3}$,
\end{itemize} we have the following (see \cite{JLR}).
We may couple $D_{n_c,p}$ and $D_{\tau_q}$ such that, w.h.p.: 
\begin{itemize}
\item $E(D_{n_c,p}) \subseteq E(D_{\tau_q})$, 
\item Every arc spanned by $V_c$ is present in $D_{n_c,p}$ independently with probability $p=\frac{2\log n_c}{3n_c}$, and
\item If $e\in E(D_{n_c,p})$ then either $e$ has color $c$ or $e\in H$ (i.e.\@ it corresponds to an arc that appears by time $m_3$).
\end{itemize}

By construction, $F,H,D_{n_c,p}$ satisfy the required conditions. Therefore Lemma \ref{reduction} implies Theorem \ref{maintm}. 
\subsection{New Setup}
The  two next sections are given in the setup of Lemma \ref{reduction} 
  (in particular, we replace $n_c$ by $n$ without further comment).
Thus we are given a vertex set $V$ of size $n$,
a 1-factor $F$ consisting of $z=\kappa\log n$ cycles,
$\kappa>0$ and a digraph $H$ of maximum in/out-degree $\Delta_H=O(\log n)$. Moreover we are given the random digraph $D_{n,p}$ where $p=\Omega \big (\frac{\log n}{n} \big)$. 

We let $\phi$ be the permutation on $V$ associated with $F$, i.e.
$E(F)=\{ (v,\phi(v)):v \in V \}$.
Furthermore we let $D^2 \sim D_{n,p'}$, $D^3 \sim D_{n,p'}$
 where $p':=\frac{\xi\log n}{n}=\min\big\{\frac{p}{3},\frac{\log n}{2n}\big\}$,
 for some $\xi = \xi(n) =\Omega(1)$. Since 
$(1-p')(1-p')\leq (1-p)$, we can couple $D_{n,p},D^2,D^3$ in such a way that  $D^2\cup D^3\subseteq D_{n,p}$.  Before proceeding we make the following observation.
\begin{lem}\label{lem5}
 W.h.p.\ $\Delta(D_{n,p'})\leq 4\log n$.
\end{lem}
\begin{proof}
\begin{align*}
\mathbb{P}\big( \Delta(D_{n,p'})\geq 4\log n \big)
&\leq 2 \cdot n  \binom{n-1}{4 \log n}{ p'}^{4\log n}
\leq  2n \bigg( \frac{en}{4\log n}\bigg)^{4\log n} \bigg(\frac{\log n}{2n}\bigg)^{4\log n} =o(1).
\qedhere
\end{align*}
\end{proof}
The proof of Lemma \ref{reduction} is splitted into two parts corresponding to \emph{Phase 2} and \emph{Phase 3} of the algorithm in \cite{apaper} that finds a Hamilton cycle in $D_{n,\frac{(1+o(1))\log n}{n}}$.
Thus we refer to the first part of Lemma \ref{reduction} as \emph{Phase 2} and to the second one as \emph{Phase 3}. As mentioned in the section  ``Finding a Hamilton Cycle"  
in \emph{Phase 2}, we  sequentially  join cycles in order to create a large one consisting of $n-o(n)$ vertices. We finish the merging of all the cycles in \emph{Phase 3}.

\section{Proof of Lemma \ref{reduction} - PHASE 2}
Let $C_1, \dots, C_z$ be the cycles in $F$ in order of decreasing size. In order to create a cycle of size at least $n-\frac{n}{\sqrt{\log n}}$ we implement the algorithm given below, denoting by $(a,b)$ the permutation transposing $a$ and $b$.

\begin{algorithm}[H]
\caption{Merge Cycles}
{\bf{Initialize:}} $\phi_1 = \phi,  E(\phi_1) = E(\phi), k=z$. \\
\While{ there exist $ 1\leq i<j \leq z$ and  $a \in V(C_i), b \in V(C_j)$ such that $ab, \phi_1^{-1}(b)\phi_1(a) \in E(D^2)\backslash E(H)$  }
{$\phi_1 \leftarrow \phi_1 \circ(a, \phi_1^{-1}(b))$\\
$E(\phi_1) \leftarrow \{ab,\phi_1^{-1}(b)\phi_1(a) \} \cup E(\phi_1)\backslash \{a\phi_1(a),\phi_1^{-1}(b)b \}$\\
$k \leftarrow k-1 $\\
Rename the cycles of $\phi_1$ as $C_1,C_2,...,C_k$ in  decreasing order of size.
}
Rename the final permutation to be $\phi_{2}$ and rename its cycles as $C_1',C_2',...,C_y'$ in decreasing order of size.
\end{algorithm}
\begin{lem}\label{bigcycle}
W.h.p.\@ $|C_1'| \geq n-\frac{n}{\sqrt{\log n}}$.
\end{lem}
\begin{proof}
Assume that after applying the algorithm above we obtain $|C_1'| < n-\frac{n}{\sqrt{\log n}}$. Set
\\ $\alpha:=\max\big\{i \in [y]: \sum_{j=1}^{i} |C_j'| <  n-\frac{n}{\sqrt{\log n}}\big\}$
,  $A:=\underset{i \in [\alpha]}{\bigcup} C_i'$ (so $|A|< n-\frac{n}{\sqrt{n}}$) and $\bar{A}:=V \backslash A.$  
 As the sequence $|C_1'|, |C_2'|,..., |C_y'|$ is decreasing,  we have 
  $$n-\frac{n}{\sqrt{\log n}} \leq \sum_{j=1}^{\alpha+1} |C_j'| \leq 2\sum_{j=1}^{\alpha} |C_j'|.$$ Hence,  $|A|=\sum_{j=1}^{i} |C_j'| \geq  \frac{n}{2}-\frac{n}{2\sqrt{\log n}}\geq \frac{n}{3}$. On the other hand $\vert \bar{A} \vert =n-|A| \geq \frac{n}{\sqrt{\log n}}.$
 Since {\emph{Merge Cycles}} ends, after performing $1\leq k\leq z$ merges with cycles $C_1',...C_y'$, we have that there do not exist $1\leq i\leq \alpha <j\leq y$ and  $a \in V(C_i'),$ $b \in V(C_j')$ such that $ab, \phi_2^{-1}(b)\phi_2(a) \in E(D^2)\backslash E(H)$. So, for every $a \in A,$ $b\in \bar{A}$; either $ab\notin E(D^2)\backslash E(H)$ or $ \phi_2^{-1}(b)\phi_2(a) \notin E(D^2)\backslash E(H)$. 
 $A,\bar{A}$ define at least $n/\sqrt{\log n}\cdot n/3$ such pairs of arcs out of which at most $2|E(H)|$ have  at least one edge in $E(H)$.
 Thus the reason that {\emph{Merge Cycles}} terminates is that for each one of those,
at most $\frac{n}{\sqrt{\log n}}\cdot\frac{n}{3}-2|E(H)|$, 
pairs of arcs at least one does not belong to $E(D_2)$. This occurs with probability at most 
$ (1-(p')^2)^{\frac{n}{\sqrt{\log n}}\cdot\frac{n}{3}-2|E(H)|}  $ (recall $D^2 \sim D_{n,p'}$).    
\vspace{3mm}
\\{\emph{Merge Cycles}} performs some number $k\leq z:=\kappa \log n$ merges. Each such merge is uniquely determined by one of its arcs (i.e.\@ either $ab$ or $\phi_1^{-1}(b)\phi_1(a)$). Hence at every execution of the while loop of {\emph{Merge Cycles}} there are at most $n(n-1)$ possible merges available.
Therefore for $0\leq k \leq z$ there are most $[n(n-1)]^k$ sequences of $k$ merges that {\emph{Merge Cycles}} may perform.
Any of those sequences may take place only if the corresponding $2k$ arcs lie in 
$E(D^2)\setminus E(H)$, so any sequence occurs with probability at most $(p')^{2k}$. 
Thus, by considering the number of merges $k$, all the possible sequences of $k$ merges that {\emph{Merge Cycles}} may perform, the probability that a given sequence the related arcs lie in $E(D^2)$ and the probability of {\emph{Merge Cycles}} terminating due to lack of additional edges after performing this exact sequence of $k$ merges, we have
\begin{align*}
\mathbb{P}\bigg(  |C_1'| < n-\frac{n}{\sqrt{\log n}} \bigg) 
&= \sum_{k=0}^z \big[n(n-1)\big]^k(p')^{2k} (1-(p')^2)^{\frac{n}{\sqrt{\log n}}\cdot\frac{n}{3}-2|E(H)|}\\
&\leq \sum_{k=0}^{\kappa\log n} (\xi\log n)^{2k} \cdot 
\exp \left\{ -\frac{\xi^2\log^2 n}{n^2}\left[ \frac{n}{\sqrt{\log n}} \cdot\frac{n}{3}-2 n \Delta_H \right] \right\}\\
& \leq (\kappa\log n+1) \cdot (\xi \log n)^{2\kappa \log n} \cdot \exp \left(  -(1+o(1))\xi^2\log^{1.5}n  \right)=o(1).
\qedhere
\end{align*}
\end{proof}

\section{Proof of Lemma \ref{reduction} - PHASE 3}

 With high probability we inherit from \emph{Phase 2} a permutation $\phi_2$ consisting of $y$ cycles, $C_1',...,,C_y'$ such that $|C_1'| \geq |C_2'| \geq ... \geq |C_y'|$,  $|C_1'| \geq  n-\frac{n}{\sqrt{\log n}}$ and $y \leq \kappa \log n$. We also inherit the edges $E(\phi_2)$ associated with the permutation $\phi_2$. 
We will use the edges in $E(D^3)$, recalling $D^3 \sim D_{n,p'}$, in order to merge one by one all the cycles with $C_1'$. At iteration $i$ of {\emph{Phase 3}} we merge $C_i'$ with the cycle $C(i-1)$. $C(i-1)$ is the output of iteration $i-1$ of {\emph{Phase 3}} and it spans $C_1',...,C_{i-1}'$. The merging of $C_i'$ with $C(i-1)$ is performed by   {\emph{FindCycle$(C(i-1),C_i', \text{ outcome})$}}. 

To merge the two cycles we start by finding arcs in $E(D^3)\setminus E(H)$ from $C_i'$ to $C(i-1)$. For every such arc we create a di-path that 
spans $V(C_i')\cup V(C(i-1))$ and uses the edges of the two cycles in addition to the selected arc.
We let the set of those di-paths be $\mathcal{P}_0^i$-we will now use the P\'osa rotations to grow $\mathcal{P}_0^i$ exponentially.
Precisely, at iteration $t$ of  {\emph{FindCycle$(C(i-1),C_i', \text{ outcome})$}} we are given a set of di-paths that spans $V(C_i')\cup V(C(i-1))$ which we denote by 
$\mathcal{P}_{t-1}^i$. For every di-path $p_r\in\mathcal{P}_{t-1}^i$ we generate every possible di-path that can be obtained from $p_r$ by a single double rotation (i.e. a two arc exchange; see Figure 2/ Section 5) with the sole condition being that the two new arcs should belong to $E(D^3) \setminus E(H)$. The new di-paths generated at iteration $t$ are added to $\mathcal{P}_{t-1}^i$ to create $\mathcal{P}_{t}^i$. We grow this collection of paths $T=\frac{\log n}{\log \log n}$ times. By this point, there are so many di-paths in     
$\mathcal{P}_{T}^i$ that a \emph{constant proportion} of all vertices have become an endpoint, and so we have a good chance to close at least one into a cycle using another arc in $E(D^3) \setminus E(H)$.
\vspace{3mm}
\\ Once more, we proceed by implicitly conditioning on all aforementioned events that are proven to occur w.h.p. 
\begin{algorithm}[H]
\caption{Phase 3 }
$C(1)=C_1'$ \\
\For{ $i=2,3,...,y$}{
    outcome $ \leftarrow$ failure \\
    suppose $C_i'= (x_{i,1},x_{i,2},...,x_{i,n_i})$ \\
    Execute FindCycle$(C(i-1),C_i', \text{ outcome})$ \\
        \If{ outcome = failure  }{
    Terminate \emph{Phase 3} }
    }
\end{algorithm}
\begin{algorithm}[H]
\caption{FindCycle$(C(i-1),C_i', \text{ outcome})$ }
Suppose $C(i-1)=(y_1,y_2,...,y_\gamma)$ .\\
Set $\mathcal{P}_0^i:=\{(x_{i,1},x_{i,2},...,x_{i,n_i},y_{j},y_{{j}+1},...,y_\gamma,y_{1},...,y_{{j}-1}): j \in [\gamma] \text{ and } x_{i,n_i}y_j \in E(D^3) \setminus E(H)\}$. \\
    \For{$t=1,...,\big\lfloor \frac{\log n}{\log \log n} \big\rfloor$}{
     Suppose $\mathcal{P}_{t-1}^i=\{p_1,p_2,...,p_s\}$ ; \\
     $\mathcal{P}_{t}^i := \mathcal{P}_{t-1}^i$\\
        \For{$r=1,...,s$}{
        Suppose $p_r=(u_1,u_2,...,u_\ell)$ \\
        For all $(a,b)$ such that $a<b$ and $(u_\ell,u_a),(u_{a-1},u_b) \in  E(D^3)\setminus E(H)$ set: \\
        $\mathcal{P}_t^i \leftarrow \mathcal{P}_t^i \cup \{(u_1,u_2,...,u_{a-1},u_b,u_{b+1},...,u_{\ell},u_a,u_{a+1},...,u_{b-1} )\}$
    } }
 Suppose $\mathcal{P}_{\lfloor \frac{\log n}{\log \log n} \rfloor}^i=\{p_1,p_2,...,p_d\}$  \hspace{5mm}\\
    \For{$k=1,...,d$}{
        Suppose $p_k=(w_1,w_2,...,w_\zeta)$ \\
        \If{$(w_\zeta,w_1) \in E(D^3)\backslash E(H) $}{
        $C(i)=(w_1,w_2,...,w_\zeta,w_1)$ \\ outcome $ \leftarrow $ success\\
        Terminate FindCycle$(C(i-1),C_i', \text{ outcome})$ }
        }
\end{algorithm}
\vspace{5mm}
With $n_1=|C'_1|$ let $C_1'=(v_1,v_2,...,v_{n_1},v_1)$. Partition $C_1'$ into $\mu_1:= \lceil \log^2n/\log \log \log n\rceil$ intervals $A_1,A_2,...$ of size $\lceil |C_1'| / \mu_1\rceil$ or  $\lfloor |C_1'| / \mu_1\rfloor$, namely $A_i=\{v_{r_{i-1}+1},v_{r_{i-1}+2},...,v_{r_{i}}\}$ for some $0=r_0< r_1<r_2<...<r_{\mu_1}= n_1$. For $I\subseteq [\mu_1]$ let $A_I:= \underset{i\in I}{\cup} A_i$, $n_I:=|A_I|$ and $B_I:= \{v \in V(C_1'): |\{u \in A_I :(v,u) \in E(D^3)\setminus E(H) \}| \leq \xi \log n/20  \}$
be the set of all vertices with much fewer than the expected number of out-neighbours to the $I$-intervals in $D^{(3)} \backslash H$.
\begin{lem}\label{ph3,1}
 W.h.p for all  $I \subseteq [\mu_1]$ with $|I|= \lfloor \mu_1 /10 \rfloor $ we have that $|B_I| \leq {n}^{1-\frac{\xi}{100}}$.
\end{lem}
\begin{proof}
For a fixed such I we have $n_I=\sum_{l \in I}|A_l| \geq |I|\lfloor {|C_1'|}/ {\mu_1} \rfloor \geq \big(\frac{\mu_1}{10}-1\big)\big( \frac{|C_1|}{\mu_1}-1\big)$. Therefore
as $n_1=|C_1'|=\left(1- \frac{1}{\sqrt{\log n}}\right) n$ we get that $ n_I=(1+o(1)) 0.1 n $. Moreover, for any vertex $v \in V$ there are at most $\Delta_H=O(\log n)$ arcs in $E(H)$ from $v$ to $A_I$.  Hence, for fixed $k$:
\begin{align*}
\mathbb{P}\big(|B_I| \geq k\big) 
& \leq \binom{n_1}{k} \mathbb{P}\Bigg[Bin\bigg(n_I -\Delta_H,\frac{\xi\log n}{n} \bigg) \leq \frac{\xi \log n}{20} \Bigg]^k\\
& \leq \bigg(\frac{en}{k}\bigg)^k \Bigg[\exp\bigg(-(1+o(1))\frac{0.5^2}{2} \frac{\xi\log n}{10}\bigg)\Bigg]^k \\
& = \bigg(\frac{e}{k}n^{1-\frac{(1+o(1))\xi}{80}} \bigg)^k\leq \bigg(\frac{e}{k}n^{1-\frac{\xi}{90}} \bigg)^k.
\end{align*}
 At the 2nd inequality we used the Chernoff bounds (3).
Thus, with $k= n^{1-\frac{\xi}{100}}$ we have
\begin{align*}
\mathbb{P}\big( \exists I \subseteq [\mu_1] : |I|= \lfloor \mu_1 /10 \rfloor ;|B_I| 
\geq n^{1-\frac{\xi}{100}}\big)
& \leq \binom{\mu_1}{{\lfloor \mu_1 /10 \rfloor}} \bigg(\frac{e}{ n^{1-\frac{\xi}{100}}}
n^{1-\frac{\xi}{90}} \bigg)^{ n^{1-\frac{\xi}{100}}} \\
& \leq 2^{\mu_1}  \bigg(e n^{-\frac{\xi}{1000}} \bigg)^{ n^{1-\frac{\xi}{100}}} = o(n^{-1}).
\qedhere
\end{align*}
\end{proof}
Next, let $\mu_2:= \lceil \frac{\log n}{\log \log \log \log n} \rceil$.
\begin{lem}\label{ph3,2}
W.h.p.\@  for every $v\in 
V$ and every  $ I \subseteq [\mu_1]$ with $|I|= \lfloor \mu_1 /10 \rfloor $,
we have $|\{b\in B_I: v\phi_2( b) \in E(D^3)\}|<\mu_2.$
\end{lem}
\begin{proof}
 For fixed $v$, $I$, and $B=\{b_1,b_2,...,b_{\mu_2}\}$,  the probability that every $v\phi_2(b_i) \in E(D^3)$ and $B \subseteq B_I$  is bounded by
$$\bigg(\frac{\xi\log n}{n}\bigg)^{\mu_2}\cdot\mathbb{P}\Bigg[Bin \bigg(n_I-\Delta_H -\mathbb{I}(v\in A_I) , \frac{\xi}{\log n} \bigg) \leq \frac{\xi \log n}{20} \bigg]^{\mu_2}
\leq \bigg(\frac{\xi\log n}{n}\bigg)^{\mu_2}\cdot n^{- \frac{\xi\mu_2}{90}}.$$
Therefore,
\begin{align*}
\mathbb{P}(\exists v,I,B \text{ as above})
& \leq n 2^{\mu_1} \binom{n }{ {\mu_2}} \bigg(\frac{\xi\log n}{n}\bigg)^{\mu_2}\cdot n^{- \frac{\xi\mu_2}{90}}
 \leq n 2^{\mu_1} \bigg(\frac{en}{\mu_2}\bigg)^{\mu_2} \bigg(\frac{\xi\log n}{n}\bigg)^{\mu_2}\cdot n^{- \frac{\xi\mu_2}{90}}\\
& \leq  \exp \bigg\{\log n+ \mu_1 \log 2+\mu_2\log\bigg(\frac{e\xi\log n}{ \mu_2}\bigg) - \frac{\xi \mu_2}{90 } \log n \bigg\}\\
& \leq \exp \big\{\Theta(\mu_1-\mu_2\log n)\big\}=o(1).
\end{align*}
\end{proof}
\begin{lem}\label{ph3,3}
 Let $0 < \alpha <1 $ be fixed. Then w.h.p.\@ there do not exist $A,B\subseteq V(C_1')$ satisfying all 3 of the following:
 \begin{enumerate}[i)]
 \item $|A| \leq \alpha_0 =\alpha e^{-3} n /\log n$,
\item $|B| \leq \alpha |A| \log n /2$
\item $ |\{(u,v) \in E(D^3) : u \in A, v \in B \} | \geq \alpha |A| \log n$.
\end{enumerate}
\end{lem}
\begin{proof}
Observe that if
there exist sets $A, B $ satisfying conditions i-iii  we may extend $B$, by  adding to it any vertices of $V(C_1')$,  to a set $B'$ of size $\alpha |A| \log n/2 $ such that the sets $A, B' $ also satisfy conditions i-iii.
Hence, if we let $\mathcal{F}$ be the event  that 
there exist sets $A, B $ satisfying conditions i-iii, then as $|V(C_1')| \leq n,$
\begin{align*}
\mathbb{P}(\mathcal{F} )
& \leq \sum_{k=1}^{\alpha_0}
\sum_{\substack{A,B\subseteq V(C_1'):\\|A|=k, |B|=\alpha k \log n/2}}\sum_{\substack{E\subseteq A\times B:\\ |E|=\alpha k \log n}}
\bigg(\frac{\xi\log n}{n} \bigg)^{\alpha k \log n}\\
&\leq \sum_{k=1}^{\alpha_0}\binom{n}{k}
\binom{n}{ {\alpha k \log n /2}} \binom{k \cdot {\alpha k \log n /2} }{{\alpha k \log n }}\cdot \bigg(\frac{\xi\log n}{n} \bigg)^{\alpha k \log n}\\
&\leq \sum_{k=1}^{\alpha_0}\Bigg\{\frac{en}{k} \Bigg[\frac{2en}{\alpha k \log n }\bigg(\frac{ek}{2}\bigg)^2\bigg( \frac{\xi \log n}{n}\bigg)^2\Bigg]^{\alpha  \log n /2}\Bigg\}^k  \\
&\leq \sum_{k=1}^{\alpha_0}\Bigg[\frac{en}{k} \bigg(\frac{k e^3 \xi \log n}{2\alpha n}\bigg)^{\alpha  \log n /2}\Bigg]^k=o(1).
\end{align*}
At the last line we used that $\xi\leq \frac{1}{2}$ and that $k\leq \alpha e^{-3} n /\log n$.
\end{proof}
We say that iteration $i$ of \emph{Phase 3} is a success if \emph{FindCycle$(C(i-1),C_i',outcome)$} merges $C(i-1)$ with $C_i'$.
To show that \emph{Phase 3} is successful  it is enough to show that for $i \in [y]$, conditioned on iteration $i-1$ of the algorithm being a success (i.e.\@ \emph{Findcycle} defines $C(i-1)$), iteration $i$ is not a success with probability   $o(\frac{1}{\log n})$ (there are $O(\log n)$ cycles to be merged). Henceforth we implicitly condition on the statements of the previous three Lemmas.
\vspace{3mm}
\\
The following three definitions will be of high significance for the rest of this section.
\begin{defn}
For $I \subseteq [\mu_1]$ set $cl(A_I):= \{ e\in E(C_1'): |e\cap V(A_I)| \geq 1 \}$,
the edges of the large cycle corresponding to the collection of intervals $I$ (together with their boundaries).
\end{defn}
\begin{defn}
We say that a path $P=(v_1,v_2,...,v_p)$ is \emph{good} if $\exists I \subseteq [\mu_1]$ with $\vert I \vert = \lfloor \mu_1 / 10 \rfloor$ and $r<s \leq \frac{p}{2}$ such that $s-r\leq \frac{p}{9}$,  $cl(A_I) \subseteq \{ v_{j}v_{j+1}: r\leq j < s\}$ and $v_p \notin B_I$ (recall $v_p \notin B_I$ if there are more than $\frac{\xi\log n}{20}$ arcs in $E(D^3)\setminus E(H)$ from $v_p$ to $A_I$).
\end{defn}
\begin{defn}
For a subgraph $S
\subseteq C(i-1),
$
set $ J_S :=\Big(\bigcup\limits_{k=2}^i V(C_k')\Big) \cup \Big(\underset{\ell \in F_{S}}{\bigcup}A_\ell\Big)$ for $F_S:=\{\ell \in [\mu_1]: cl(A_{\ell}) \not\subseteq E(S) \}$.
This $J_S$ should be considered as a set of junk: we want to restrict ourselves to only trying more rotations using the intervals $\ell \in [\mu_1]$ preserved from the original large cycle $C_1'$ which are still wholly contained in $S$ (i.e.\ were not broken by a previous rotation). Certainly therefore we want to avoid any vertices leftover from the smaller cycles $C_k'$ that have previously been merged.
\end{defn}
\begin{lem}
Suppose $S$ is a good path that satisfies $S\in  {\mathcal{P}}_t^i$ for some $0\leq t \leq \frac{\log n}{\log \log n}.$
Then $|J_S| = o(n)$.
\end{lem}
\begin{proof}
 To merge $C(i-1)$ with $C_i'$, we start by joining the two cycles using an edge in $E(D^3)\setminus E(H)$, then delete an edge from each cycle to create a path.
Thereafter, in order to create a new path from a given one, we perform double rotations (defined in section Finding Hamilton cycles - Overview). Every
double rotation
involves removing two edges from the current path and adding two edges from $E(D^3)\setminus E(H)$.
As \emph{FindCycle$(\cdot)$} performs $\leq \frac{\log n}{\log \log n }$ rounds of double rotations,
$\vert E(C(i-1)) \backslash E(S) \vert\leq 1+ 2\cdot\frac{\log n}{\log \log n } $.
Similarly, $\vert E(C(k-1)) \backslash E(C(k)) \vert\leq 1+ 2\cdot\frac{\log n}{\log \log n } $ for every $2 \leq k <i$. Thus, as $i \leq \log n$, we have
$$
\vert F_S \vert
\leq 2\vert E(C_1')\backslash E(S)\vert = 2\vert E(C(1))\backslash E(S) \vert  \leq 4\log n \cdot\bigg(1+ 2\cdot\frac{\log n}{\log \log n } \bigg) = o(\mu_1).$$
(At the first inequality, we used that each removed $e \in E(C_1')$ was in $\leq 2$ of the $cl(A_\ell)$'s).
Therefore,
\begin{align*} 
|J_S| &\leq \sum_{k=2}^{i}|V(C_k')| +\sum_{\ell \in F_S} \vert  A_\ell \vert
\leq o(n)+ o(\mu_1) \cdot( n/\mu_1+1) =o(n).\qedhere
\end{align*}
\end{proof}
\begin{defn}
Let $i \in [y]$ and $x \in V(C_i')$. For $t \leq   \frac{ \log n}{\log \log n}$ we define  $\mathcal{GP}_{t}^i$ to be the set of all good paths that are contained in $\mathcal{P}_t^i$. Furthermore let $ENDG_{t}^i$ be the set of endpoints of paths in $\mathcal{GP}_t^i$.
\end{defn}
\begin{lem}\label{startpath}
 For $i \in [y]$, conditioned on iteration $i-1$ being a success, $ \mathbb{P}(\mathcal{GP}_{t}^i \neq \emptyset) \geq 1-o(n^{-\frac{\xi}{2}}) $.
\end{lem}
\begin{proof}
 Let $C(i-1)=\{u_1,u_2,...,u_\gamma,u_1\}$. Partition $C(i-1)$ into 9 blocks/subpaths $S_1,S_2,...,S_9$ of near-equal length by setting, for each $\ell\in [9]$, $S_\ell := \{u_{\lfloor\frac{\ell-1}{9}\cdot \gamma \rfloor+1},..., u_{\lfloor\frac{\ell}{9}\cdot \gamma\rfloor}\}$.
Note every $|J_{S_\ell} \cap S_\ell| \leq \vert J_{C(i-1)}\vert +2 = o(n)$, so
\begin{equation}\label{eq:ax}
 \sum_{\substack{i \in [\mu_1]
\\ cl(A_i) \subseteq E(S_\ell)}}
 \left\lvert A_i \right\rvert
 =| S_\ell \backslash J_{S_\ell} |=|S_\ell|
-o(n)\geq \bigg|\frac{C_1'}{9}\bigg|-1-o(n)=\big(1-o(1)\big)\frac{n}{9}.
\end{equation}
For every $\ell \in [9]$, let ${I_\ell}'=\{ i \in  [\mu_1]: 
cl(A_i) \subseteq E(S_\ell)\}$. (\ref{eq:ax}) implies that $|I_\ell'| \geq \mu_1/10$. Thus we may let $I_\ell\subseteq I_\ell'$ be the set of the  
$\lfloor \mu_1 / 10 \rfloor$ smallest elements of $I_\ell'$.

Recall the notation $C_i'= \{x_{i,1}x_{i,2},...,x_{i,n_i},x_{i,1 } \}$.
$\mathcal{GP}_0^i$   in non-empty if there exists 
an arc $(x_{i,n_i},u_a) \in E(D^3)\setminus E(H)$ for some $a \in [\gamma]$
such that
\begin{enumerate}[(i)]
\item $u_a\in A_{I_\ell}$ for some $\ell \in [9]$, and
\item $\phi^{-1}_2(u_a) \notin B_{I_{1}} \cup B_{I_{2}} \cup ... \cup B_{I_{9}}$.
\end{enumerate}
Indeed let $P=\{x_{i,1},...,x_{i,n_i}, u_a,u_{a+1},...,u_\gamma,u_1,...,$ $u_{a-1}\} $ be such a path. Observe  that $\exists j \in [9]$ such that $S_j$ defined above is found in the interior of the first half of $P$ (here we only needed that $C(i-1)$ was split into at least 5 blocks). In addition $S_j$ 
consists of $\frac{n}{9}-o(n)$ consecutive vertices in $C(i-1)$ hence in $P$.
Thus since $I_j\subseteq I_j'\subsetneq S_j$, $I:=I_j$ is a witness to the goodness of path $P$. Furthermore 
$u_a\in A_{I_\ell}$ implies that $ (\phi^{-1}_2(u_a),u_a) \in E(C(i-1))$ and therefore $ \phi^{-1}_2(u_a)=u_{a-1}$. Finally since the endpoint of $P$,
$u_{a-1}=\phi^{-1}_2(u_a)\notin B_{I_{1}} \cup B_{I_{2}} \cup ... \cup B_{I_{9}}$ we have that all the conditions for $P$ to be good are met.

Lemma \ref{ph3,1} implies that the number of vertices $u_a$ satisfying both conditions (i) and (ii) is $(1+o(1))0.9n$. Since we do not examine the arcs in $\{x_{i,n_1}\}\times V(C_1')$ that are found in $E(D^3)$ until  we execute the $i$-th iteration of \emph{Phase 3},  we have that  any arc
in $\{x_{i,n_1}\}\times V \big(\underset{i\in [\ell]}{\cup} A_{I_\ell}\big)$ not found in $E(H)$ belongs to $E(D^3)$ with probability $p'=\frac{\xi \log n}{n}$. Pause for a moment to recall that every vertex has at most $\Delta_H=O(\log n)$ out-arcs in $E(H)$ that we cannot use. Thus, given that iteration $i-1$ is a success,
 the probability of the event $\{\mathcal{GP}_0^i = \emptyset\}$ is bounded above by
\begin{align*}
\mathbb{P}\bigg\{Bin\big[((1+o(1))0.9n -\Delta_H,p'\big]=0\bigg\}
&\leq (1-p')^{(1+o(1))0.9n}
\leq e^{-(1+o(1))0.9p'n } = o(n^{-\frac{\xi}{2}}).
\qedhere
\end{align*}
\end{proof}

We will use the endpoints of good paths in order to lower bound the number of distinct endpoints of paths created at some iteration of \emph{Phase 3}. The advantage of good paths is that their endpoints have many arcs towards earlier vertices of the path, whose predecessors in turn have many arcs to vertices nearer the end of the path. Hence, we expect the number of paths originating from a specific good path after an iteration of \emph{Phase 3} to be large. Note that for any $i \in [y]$ all the paths that are constructed during \emph{FindCycle$(C(i-1),C_i',outcome)$} have the same starting point, namely $x_{i,1}$.

\begin{lem}\label{last}
Let $i \in [y]$  be such that  $\mathcal{GP}_{t}^i \neq \emptyset$. Then, w.h.p.\@ for $t \leq   \frac{ \log n}{\log \log n}-1,$
$$ |ENDG_{t}^i| \leq \frac{\xi n}{84 e^3 \log^2 n} \hspace{5mm} \text{ implies }  \hspace{5mm}  \bigg(\frac{\xi \log n}{42}\bigg)^2|ENDG_{t}^i|\leq|ENDG_{t+1}^i|.$$
\end{lem}
\begin{proof}
For $t \leq   \frac{\log n}{\log \log n}-1$ let $P= (u_1,u_2,...,u_p) \in \mathcal{GP}_t^i$ and $r_P,s_P,I_P$ be as in the definition of a good path.
Partition $P$ into 9 sub-paths $S_{1,P},S_{2,P},...,S_{9,P}$ containing $A_{I_{1,P}},A_{I_{2,P}},...A_{I_{9,P}}$ as is done earlier in Lemma \ref{startpath}. Set
$$ H_1(P)= \{ u_j \in P: u_pu_j \in E(D^3)\setminus E(H), u_j \in A_{I_P} \text{ and } u_{j-1} \notin B_{I_{9,P}}   \} $$
and
$$H_2(P)= \{ u_{j-1}: u_j \in H_1(P)\}.$$
Since $P$ is a good path we have that $u_p \notin B_{I_P}$. Therefore $u_p$ has at least $\frac{\xi\log n}{20}$ neighbours in $A_{I_P}$ out of which at most $\mu_2$ have their predecessor in $B_{I_{9,P}}$ (see Lemma \ref{ph3,2}). Hence  we have that
\begin{align}
\vert H_2(P) \vert =\vert H_1(P) \vert  &\geq \frac{\xi\log n}{20} - \mu_2  \geq \frac{\xi\log n}{21}.
\end{align}
Furthermore, if $r_P <\frac{p}{9}+1$ for each  $u \in H_2(P)$ set,
$$ H_3(P,u)= \{ u_\ell \in P: uu_\ell \in  E(D^3)\setminus E(H), u_\ell \in A_{I_{9,P}} \text{ and } u_{\ell-1} \notin B_{I_{3,P}}   \} .$$
Otherwise, set
$$ H_3(P,u)= \{ u_\ell \in P: uu_\ell \in  E(D^3)\setminus E(H), u_\ell \in A_{I_{9,P}} \text{ and } u_{\ell-1} \notin B_{I_{1,P}}   \} .$$
Finally in both of the above cases set
$$H_4(P,u)= \{ u_{\ell-1}: u_\ell \in H_3(P,u)\}.$$
As before, from  $H_2(P) \cap  B_{I_{9,P}}= \emptyset$ together with Lemma \ref{ph3,2} we have that, for all $u \in H_2(P)$,
\begin{align}
\vert H_4(P,u) \vert =\vert H_3(P,u) \vert  &\geq \frac{\xi \log n}{20} - \mu_2  \geq \frac{\xi \log n}{21}.
\end{align}
\\ Finally for $k\in\{1,2\}$ and $m\in \{3,4\}$ set,
$$\mathcal{H}_k:= \underset{P \in \mathcal{GP}_{t}^i}{\bigcup}H_k(P) \hspace{20mm}
\mathcal{H}_m:= \underset{P \in \mathcal{GP}_{t}^i}{\bigcup}\bigg\{\underset{v\in H_2(P)}{\bigcup}H_m(P,v)\bigg\}.$$
\vspace{5mm}
\\\emph{Claim:} $\mathcal{H}_4 \subseteq ENDG_{t+1}^i$.
\vspace{5mm}
\\\emph{Proof of the claim:} Indeed, suppose that $r_P<\frac{p}{9}+1$ and $u_{k-1} \in \mathcal{H}_4,$ i.e.\@ there are $j$ and $k$ such that
$$ u_pu_j,  u_{j-1}u_k \in F_c^3, \hspace{5mm} u_j \in A_{I_P},\hspace{5mm} u_k \in A_{I_{9,P}},\hspace{5mm} u_{j-1} \notin B_{I_{9,P}}  \text{ and } u_{k-1} \notin B_{I_{3,P}}.$$
Then, $r_P\leq j\leq s_P\leq \frac{p}{2}\leq k$ and hence a double rotation on $P$ using the edges  $u_pu_j,  u_{j-1}u_k$ will result in
the path $P'=(u_1,u_2,..u_{j-1}u_k,u_{k+1},...,u_p,u_j,u_{j+1},...,u_{k-1})$.
So in showing that $u_{k-1} \in ENDG^i_{t+1}$ it suffices to show that $P'$ is a good path with $I_{P'}=I_{3,P}$. To see this first note  $u_{k-1}\notin B_{ I_{3,P} }$.
Secondly  $cl(A_{I_{3,P}}) \subseteq P'$ as $cl(A_{I_{3,P}}) \subseteq P$ and no edge of $cl(A_{I_{3,P}})$ was deleted in a double rotation.
Thirdly if we let $r',s'$ to be respectively the smallest and largest indices of vertices in $A_{I_{3,P}} (=A_{I_{P'}})$ in the path $P$ then $(s'+1)-(r'-1)\leq \frac{\vert P' \vert}{9}(=\frac{p}{9})$ as $cl(A_{I_{3,P}})\subseteq E( S_{3,P})$.
This implies that $ cl(A_{I_{P'}})
\subseteq \{ u_ju_{j+1}: (r'-1)+(p-k+1)\leq j < (s'+1)+(p-k+1)\}$
and that $[(s'+1)-(p-k+1)]-[(r'-1)-(p-k+1)] \leq \frac{p}{9}$. Finally as $u_k \in A_{I_{9,P}}$ and $u_{s'} \in  A_{I_{3,P}}$, we get that $p-k\leq \frac{p}{9}$ and $(s'+1) \leq \frac{p}{3}$. Hence $(s'+1)+(p-k+1)<\frac{p}{2}.$

In the case that
$r_P>\frac{p}{9}$ and $u_{k-1} \in \mathcal{H}_4,$ the goodness of $p'$ (now with $I_{P'}=I_{1,P}$) follows from the same reasoning with the only difference that the vertices in $A_{I_{P'}}$ hold the same positions in both paths. Thus in both cases $P'$ is good, proving the claim.
\vspace{5mm}
\\Suppose that $ |ENDG_{t}^i| \leq \frac{\xi n}{84e^3\log^2 n}$. To make sure that the endpoints of good paths in $\mathcal{GP}_{t+1}^i$ do not coincide too often  we apply Lemma \ref{ph3,3} with $\alpha=\frac{\xi}{21}, A=ENDG_{t}^i ,B=\mathcal{H}_1$. Recall for every good path there are at least $\frac{\xi\log n}{21}$ edges in $E(D^3)\backslash E(H)$ from its endpoint that lie in 
$A$ to vertices in $B=\mathcal{H}_1$. So by summing over a maximal set of paths with distinct endpoints we get that there are at least $\frac{\xi}{21}|A|\log n$ arcs from $A$ to $B$. Hence as
$\vert A \vert  \leq \frac{\xi n}{84e^3\log^2n}\leq \alpha e^{-3} n /\log n $ in the Lemma \ref{ph3,3} condition ii) must not be satisfied. Moreover Lemma \ref{lem5} implies that w.h.p.\@ there are at most $\Delta(D^3)\vert A \vert \leq 4\log n |A|$ arcs from $A$ to $B$. Therefore,
$$ \frac{\xi \log n}{42} |ENDG_{t}^i| \leq |\mathcal{H}_1|=|\mathcal{H}_2| \leq 4 \log n |ENDG_{t}^i| \leq \frac{\xi n}{21e^3\log n}.$$
Similarly by reapplying Lemma \ref{ph3,3} with  $\alpha=\frac{\xi}{21}, A=\mathcal{H}_2,B=\mathcal{H}_3$  we have that,
\begin{align*}
\bigg(\frac{\xi \log n}{42}\bigg)^2|ENDG_{t}^i| \leq \frac{\xi \log n}{42}  |\mathcal{H}_2|&\leq |\mathcal{H}_3|=|\mathcal{H}_4|   \leq |ENDG_{t+1}^i|.
\qedhere
\end{align*}
\end{proof}
Summarising, the two last lemmas give us that conditioned on phase $i-1$ being a success,  $1 \leq \vert ENDG_{0}^i \vert$ with probability at least $1-o(n^{-\frac{\xi}{2}})$.
 Furthermore since $n\leq \big(\frac{\xi \log n}{42}\big)^{ \frac{1.8 \log n}{\log \log n}} $ 
 the integer $t_f:=\min \big\{j:\big(\frac{\xi \log n}{42}\big)^{2j}\geq \frac{\xi n}{84 e^3 \log^2n}  \big\} $ is less than $\frac{0.9 \log n}{\log \log n}$
 and satisfies, due to Lemma \ref{last}, $|ENDG_{t_f}^i| \geq \frac{\xi n}{84 e^3 \log^2n}$. Thus by applying the same argument as in the previous lemma to a subset $F$ of  $ENDG_{t_f}^i$ of size $\frac{\xi n}{84 e^3 \log^2n}$ and to the set of paths in $\mathcal{GP}_{t_f}^i$ with endpoints in $F$ we have that
  $$\beta n=\bigg(\frac{\xi \log n}{42}\bigg)^2 \cdot \frac{\xi n}{84 e^3 \log^2n} \leq |ENDG_{t_f+1}^i(v)|$$ for some constant $\beta >0$.
  Recall that all the paths in
 $\mathcal{GP}_{t_f+1}^i$ start from the same vertex $x_{1,i} \in V(C_i')$ and
 that $\mathcal{GP}_{t_f+1}^i \subseteq \mathcal{P}^i_{\lfloor{\frac{\log n}{\log \log n}}\rfloor}$.
Since we do not examine the arcs going into $x_{i,1}$ until the very end of the $i$-th iteration of \emph{Phase 3}, \emph{after} conditioning on iteration $i-1$ of \emph{Phase 3} being a success every arc in $V(C_1')\times \{x_{i,1}\} \backslash E(H)$ still belongs to $E(D^3)$ with probability $p'$
.
Hence, the probability of iteration $i$ of \emph{Phase 3} not being a success conditioned on iteration $i-1$
 is bounded by
 $$o\big(n^{-\frac{\xi}{2}}\big)+\mathbb{P}\big[Bin(\beta n-\Delta_H,p')=0\big]\leq o\big(n^{-\frac{\xi}{2}}\big)+(1-p')^{\beta n-O(\log n)}=o(n^{-\epsilon}),$$
 for some $\epsilon >0$. As we merge cycles at most 
 $y \leq \kappa \log n$ times, \emph{Phase 3} succeeds in merging all the cycles into one with probability $1-o(n^{-\epsilon} \cdot \kappa \log n)=1-o(1)$. Finally observe that during phases 2 and 3 we use edges only in $(E(D^2) \cup E(D^3))\backslash E(H)$ which completes the proof of Lemma \ref{reduction}.

\begin{appendices}
\section{Proof of Theorem \ref{secondtm}}

Theorem \ref{secondtm} can be proven in an almost identical fashion to  Theorem \ref{maintm}. As the proof of  Theorem \ref{maintm} is somewhat lengthy with many technicalities  we are only going to present a sketch of the proof of Theorem \ref{secondtm} where we highlight substantial differences. 
\subsection{Some notation}
Write $\tau'$ for $\tau_{2q}'$, that is, the hitting time $\tau'$ for when $G_{\tau'}$ first has minimum degree $2q$.
Recall that $\tau' \in [m_\ell, m_u]$ w.h.p.\@, where now instead $m_u, m_\ell: = n \log n + (2q-1)n \log \log n \pm \omega(n)$, which as before underpins the same computations in this setting.
\begin{notn} For $u,v\in V_n$
we say the we orient the edge $uv$ $+u$ or equivalently $-v$ if we orient it from $u$ to $v$.  
\end{notn}
\begin{defn}
For $v \in V_n$, $c \in [q]$ and $t \in \{0,1,...,n(n-1)/2\}$ we define the quantities $d^+_t(v,c),$  $d^-_t(v,c)$, $d^+_t(v)$,  $d^-_t(v)$,  $d^+(v)$, $d^-(v)$ and the sets  $C^{+}_v(t)$, $C^{-}_v(t)$ as in the subsection \ref{notation}.
\end{defn}
We are now interested in assigning every (color,direction)-pair to the edges adjacent to a vertex.  
\begin{defn}
For $t \in \{0,1,...,\tau' \}$ we set $FULL_t:=\{v \in V_n: C^{+}_v(t)\cup C^{-}_v(t) \neq \emptyset \}$ (i.e.  the set of vertices that at time $t$ have out-degree and in-degree in each color at least one). 
\end{defn}

As before, since some vertices will only have degree $2q$ in the whole graph $G_{\tau'}$, when a new edge $e_t$ appears we need to prioritize any vertices that aren't yet in $FULL_{t-1}$:

\begin{algorithm}
\caption{ColorGreedy2($u,v,t$)}
\eIf{ $u \notin FULL_{t-1}$ or $v \notin FULL_{t-1}$}{
Orient and color the edge $uv$ by an orientation and a color that is chosen uniformly at random from 
$\{ (x,c)\in \{+1,-1\} \times [q]: \mathbb{I}\big(d^{sign(x)}_{t-1}(u,c)=0\big)+\mathbb{I}\big( d^{sign(-x)}_{t-1}(v,c)=0\big)\geq 1 \}$.
}{ 
Orient $uv$ uniformly at random, \\
color  $uv$ with a  color that is chosen uniformly at random from  [q].
}
\end{algorithm}
For $i \in \{0,1,2,3\}$ we still take $m_i=i \cdot e^{-q\cdot 10^{4}} n \log n$.
\begin{algorithm}[H]
\caption{COL-ORIENT}
\For{ $t=1,...,m_1$}{
let $e_t=uv$ \\
Execute ColorGreedy2($u,v,t$). }
For $v\in V_n$ set $c^+(v)=1, c^-(v)=1$.\\
\For{ $t=m_1+1,...,m_2$}{
let $e_t=uv$ \\
\eIf{ $ u\notin FULL_{t-1} \text{ or } v \notin FULL_{t-1}$}{
Execute ColorGreedy2($u,v,t$). }
{ Choose $w\in \{u,v\}$ uniformly at random; orient $e_t$ +$w$. \\
Color  the arc $e_t$ by the color $c$ that satisfies $c\equiv c^+(w) \mod q$,\\
$c^+(w) \leftarrow c^+(w)+1$.}
}
\For{$t=m_2+1,m_2+2,...,m_3$}{
let $e_t=uv$ \\
\eIf{ $ u\notin FULL_{t-1} \text{ or } v \notin FULL_{t-1}$}{
Execute ColorGreedy2($u,v,t$). }
{ Choose $w\in \{u,v\}$ uniformly at random; orient $e_t$ -$w$. \\
Color  the arc $e_t$ by the color $c$ that satisfies $c\equiv c^-(w) \mod q$,\\
$c^-(w) \leftarrow c^-(w)+1$.}
}
For $i \in \{1,2,3\}, * \in \{+,-\}$ set $B^*_i:=\{ v \in V_n: d^*_{m_i}(v)-d^*_{m_{i-1}} \leq e^{-q \cdot 10^{6}}, \log n\}$.
\\Furthermore, set $BAD:=B_1^+ \cup B_1^- \cup B_2^+ \cup B_3^-$ and $E'=\emptyset$.\\
\For{ $t=m_3+1,...,\tau'$}{
let $e_t=uv$ \\
    \uIf{ $ u\notin FULL_{t-1} \text{ or } v \notin FULL_{t-1}$}{
Execute ColorGreedy2($u,v,t$). }
    \uElseIf{ $u \in BAD$ or $v \in BAD$}{
    Choose uniformly at random $(x,c)\in \{+1,-1\} \times [q]$ from those that minimize the expression  $d^{sign(x)}_t(u,c)\mathbb{I}(u \in BAD)+  d^{sign(-x)}_t(v,c)\mathbb{I}(v \in BAD)$.\\   
    Color the edge $uv$ by color $c$ and   orient it $sign(x)u$. }
    \Else{ Execute ColorGreedy2($u,v,t$). \\
    Add $uv$ to $E'$.}
    }
\end{algorithm}
\begin{rem}
As in algorithm $COL$, if for some $e_t=uv$, $u\notin FULL_{t-1}$ or $v\notin FULL_{t-1}$ then any $(x,c)\in \{+1,-1\} \times [q]$ that satisfies $d^{sign(x)}_{t-1}(u,c)=0$ or  $d^{sign(x)}_{t-1}(v,c)=0$ may be chosen in the assignment of orientation and color to $uv$ with probability at least $\frac{1}{2q}$.
\end{rem}
It is easy to see that the Lemmas in Section 3 have an undirected version which can be proven in the same way. 
On the other hand in order to prove that $COL$-$ORIENT$  assigns directions and colors to the edges such that $\forall v\in V_n$ and $\forall c\in[q]$ $d_{\tau'}^+(v,c),  d_{\tau'}^-(v,c)\geq 1$  we  make a small additional calculation.   Indeed, recall that the first step in the proof of Theorem \ref{tm2} was to define the set $A_L^+(v)$ of arcs to $N_L^+(v)$ which a particular vertex $v$ needs to provide an out-arc in each color. To help this $v$ get priority for enough of the colors assigned to $A_L^+(v)$, each $w \in N_L^+(v)$ had a set $B_v^-(w)$ of $\Omega(\log n)$ arcs supporting this $w$'s own in-edge-coloring needs.
Most significantly, any pair of these supporting arc-sets 
$B_v^-(w)$ and $B_v^-(w')$ were disjoint since an edge can't have both $w$ and $w'$ as in-vertices.
Disconcertingly, in the current now undirected case there could still be triangles in $G_{\tau'}$. Nevertheless, the existence of the corresponding sets (defined below) will result from the following lemma. 
\begin{lem}\label{alt}
W.h.p.\@ $G_{\tau'}$ does not contain a cycle of length 4 with a chord. Hence for $v\in V_n$, if $N(v)$ denotes the neighbours of $v$ in $G_{\tau'}$, at time $\tau'$, every $w\in N(v)$  has at most one neighbour in $N(v)$. 
\end{lem}
\begin{proof}
Using the undirected version of (2) with $p_u$ defined by $p_u\binom{n}{2} = m_u$, together with Markov's inequality and the fact that almost surely $\tau'\leq m_u$, we get that the probability that such a subgraph exists is bounded by
\begin{align*}
& \binom{n}{4} 4! p_u^{5}\leq  24n^4 \bigg(\frac{2\log n}{n} \bigg)^{5} =o(1).
\qedhere
\end{align*}
\end{proof}
As before, fix $v\in V_n$. Denote by $N_L(v)$ the neighbours of $v$ in $G_{\tau'}$ with more than $\frac{\log n}{100}$ neighbours in $G_{\tau'}$ and $A_L(v)$ the set of edges arising from $N_L(v)$ (i.e. $A_L(v):=\{vw \in E_{\tau'}:w \in N_L(v) \}$). 
\\
For each $w \in N_L(v)$, Lemma \ref{alt} shows us at most one edge was from $w$ to another vertex in $N(v)$, so it follows $\vert \big\{ wx\in E_{\tau'} :x\notin N(v)\cup\{v\} \big\} \vert \geq \frac{\log n}{100}-2$.
This means we can define $B_{v}(w)$
to be any subset of size $\frac{\log n}{100} -2$ of these arcs.
 Finally we let $A_v(w):= B_v(w) \cup \{vw\}$, so that
$|A_v(w)|=\frac{\log n}{100} - 1$.
Thereafter we can couple the coloring/orientation process of two graphs $G_{\tau'}^{(1)}$, $G_{\tau'}^{(2)}$ in parallel (as per Section 4, $G_{\tau'}^{(2)}$ has the same distribution as $COL$-$ORIENT$, but $G_{\tau'}^{(1)}$ has more randomness among the edges distance 1 away from $v$).
Specifically, in the spirit of \emph{$COL1(v)$},
we define the graphs $G_{\tau'}^{(1)}, G_{\tau'}^{(2)}$ via the algorithm $COL$-$ORIENT1(v)$ given below:
 
\begin{algorithm}[H]
\caption{COL-ORIENT1(v)}
\For{ $t=1,...,\tau'$}{
    let $e_t=ab$ \\
   \eIf{ $e_t \in \underset{w \in N_L^+(v)}{\bigcup}B_v(w) $}{
   Choose $(x,c)$ from $\{+1,-1\}\times [q]$ uniformly at random; \\
        \eIf{ $c \in C^{sign(x)}_{2,a}(t-1) \cup C^{sign(-x)}_{2,b}(t-1)$ }{ in both $G_{\tau'}^{(1)}$, $G_{{\tau'}}^{(2)}$ color $e_t$ with color $c$ and orient it $sign(x) a$.}
        { Color $e_t$ in  $G_{{\tau'}}^{(1)}$  with color $c$ and orient it $sign(x)a$,\\
          to assign and orient $e_t$ in $G_{{\tau'}}^{(2)}$ execute step $t$ of COL-ORIENT. 
          }
     }
     {To color and orient $e_t$ in $G_{{\tau'}}^{(2)}$ execute step $t$ of COL-ORIENT,\\
      assign to $e_t$ in $G_{{\tau'}}^{(1)}$ the same color and direction as in $G_{{\tau'}}^{(2)}$.
}}
\end{algorithm}
Note that the sufficient conditions corresponding to those needed while proving Lemmas \ref{lem12} and \ref{lem13} are met (see Remark \ref{sss}). Specifically, the $\{A_v(w)\}$ are disjoint, each $A_v(w)$ has size $\Omega\big({\log n}\big)$, and for every $w\in N_L(v)$ every edge in $B_v(w)$ is colored and oriented by $COL$-$ORIENT1(v)$ independently and uniformly at random in $G_{\tau'}^{(1)}$.  Hence we can proceed analogously to Section 4.
\vspace{3mm}
\\For the part of the proof corresponding to Section 5 we can define $BAD$ as the vertices that are adjacent to few edges in one of the desired directions in $E_{m_1}, E_{m_2}\backslash E_{m_1},$ or $E_{m_3}\backslash E_{m_2}$. Then by repeating the calculations, with the undirected random graph process in place of the directed random graph process, we can bound the size of $BAD$ by
$n^{1-\delta}$ for some constant $\delta=\delta(q)>0$. As such, we can proceed and ``hide'' $BAD$ with a similar algorithm in order to form $G_c$. The calculations found in section 8 (corresponding to \emph{Phase 1}),
and the reduction to Lemma \ref{reduction}, are all the same as before.
\end{appendices}

\end{document}